\documentclass{article}

\usepackage[utf8]{inputenc} 
\usepackage{amsfonts}
\usepackage{amsmath}
\usepackage{amsthm} 
\usepackage{amssymb}
\usepackage{color}
\usepackage{graphicx}
\usepackage{setspace}
\usepackage{enumerate}
\usepackage{booktabs}
\usepackage{subfigure}
\usepackage{pstricks}
\usepackage{url}
\usepackage{hyperref}
\usepackage{mathtools}
\usepackage{enumitem}
\usepackage{cancel}
\usepackage[ruled, boxed, noend, noline, linesnumbered, norelsize]{algorithm2e}
\usepackage[normalem]{ulem}
\usepackage{multicol}
\hypersetup{
	colorlinks,
	linkcolor={red!50!black},
	citecolor={blue!50!black},
	urlcolor={blue!80!black}
}

\newcommand{\norm}[1]{\left\|#1\right\|}
\newcommand{\interior}{\mathrm{int}\,}

\newcommand{\inProd}[2]{\langle #1 , #2 \rangle }

\newcommand{\feasPs}{\mathcal{F}^{\stdMap}_{\text{scaled}}}

\newcommand{\PSDcone}[1]{{\mathcal{S}^{#1}_+}}

\newcommand{\stdMap}{ {\mathcal{A}}}

\newcommand{\stdCone}{ {\mathcal{K}}}

\newcommand{\stdInt}{ {e}}

\newcommand{\T}{\top\hspace{-1pt}}               

\newcommand{\Vol}{{\mathrm{ vol } \,}}    
\newcommand{\jAlg}{\mathcal{E}}
\renewcommand{\Re}{\mathbb{R}}
\newcommand{\const}{\frac{1}{\rho r}}    
\newcommand{\constInv}{\rho r} 
\newcommand{\SOC}[1]{\mathcal{L}_{#1}}
\newcommand{\jProd}[2]{ {#1 \circ #2 } }
\newcommand{\eig}{\lambda}
\newcommand{\Qr}[1]{Q_{#1}}
\newcommand{\SC}{1}

\DeclarePairedDelimiter\abs{\lvert}{\rvert}%

\renewcommand{\S}{\mathcal{S}}                    
\newcommand{\tr}{\mathrm{tr}\,}    
\newcommand{\RNum}[1]{\uppercase\expandafter{\romannumeral #1\relax}}
\newtheorem{definition}{Definition}
\newtheorem{lemma}[definition]{Lemma}
\newtheorem{proposition}[definition]{Proposition}

\newtheorem{corollary}[definition]{Corollary}

\newtheorem{theorem}[definition]{Theorem}

\title{An extension of Chubanov's algorithm to symmetric cones}
\date{December 2016 (Revised: September 2017)}

\author{Bruno F. Louren\c{c}o
	\thanks{Department of Computer and Information Science, Faculty of Science and Technology, Seikei
		University, 3-3-1 Kichijojikitamachi, Musashino-shi, Tokyo 180-8633, Japan.
		(Email: lourenco@st.seikei.ac.jp)}
	\and
	Tomonari Kitahara\thanks{
		Graduate School of Decision Science and Technology, Tokyo Institute of Technology,
		2-12-1-W9-62, Ookayama, Meguro-ku, Tokyo 152-8552, Japan. 
		(E-mail: kitahara.t.ab@m.titech.ac.jp)
		}
	\and
	Masakazu Muramatsu\thanks{
		Department of Computer Science, The University of Electro-Communications 1-5-1 Chofugaoka, Chofu-shi, Tokyo, 182-8585 Japan. (E-mail: muramatu@cs.uec.ac.jp)
	}
	\and       
	Takashi Tsuchiya
	\thanks{
		National Graduate Institute for Policy Studies 7-22-1 Roppongi, Minato-ku, Tokyo 106-8677, Japan. (E-mail: tsuchiya@grips.ac.jp)
	}
}

\begin{document}
\maketitle
\begin{abstract}
	In this work we present an extension of Chubanov's algorithm to 
	the case of homogeneous feasibility problems over a symmetric cone $\stdCone$. 
	As in Chubanov's method for linear feasibility problems, the algorithm consists of a \emph{basic procedure} and a step where the solutions are confined to the intersection of a half-space 
	and $\stdCone$. Following an earlier work by Kitahara and Tsuchiya on second order cone feasibility problems, progress is measured through the volumes of those 
	intersections: when they become sufficiently small, we know it is time to stop.
	We never have to explicitly compute the volumes, it is only necessary to keep 
	track of the reductions between iterations. We show this is enough  to obtain 
	concrete upper bounds to the minimum eigenvalues of a scaled version of the original 
	feasibility problem. Another distinguishing feature of our approach is the 
	usage of a spectral norm that takes into account the way that $\stdCone$ is decomposed 
	as simple cones. 
	In several key cases, including semidefinite programming and second order cone programming, these norms  make it possible to obtain better complexity bounds for the basic procedure 	when compared to a recent approach by Pe\~{n}a and Soheili. 
	Finally, in the appendix, we present a translation of the algorithm to the homogeneous feasibility problem in semidefinite programming.
	
	
\end{abstract}

\section{Introduction}\label{sec:int}
Consider the following feasibility problem:	
	\noindent\begin{align}
	{\mathrm{find}} & \quad x \tag{P} \label{eq:feas_p} \\ 
	\mbox{subject to} & \quad \stdMap x = 0  \notag \\
	& \quad x \in \interior \stdCone, \notag
	\end{align} 
where $\stdCone$ is a \emph{symmetric cone} in a finite dimensional 
space $\jAlg$ of dimension $d$,  $\interior \stdCone$ denotes the interior of $\stdCone$ and 
$\stdMap$ is a linear map. 
The problem \eqref{eq:feas_p} is a special case of \emph{symmetric cone programming} (SCPs), a broad framework 
that contains linear programming (LP), second order cone programming (SOCP) and semidefinite programming (SDP). In particular,  among the symmetric cones, we have the nonnegative orthant $\Re^n_+$, the second order cone (Lorentz cone) $\SOC{n}$, the positive semidefinite cone $\PSDcone{n}$ and any direct product of those.

Besides  providing an unified framework for several important classes of problems,
symmetric cones are very rich in structure, making it possible to use several 
useful linear algebraic concepts such as eigenvalues in far broader context.  
Typically,  symmetric cone programs are solved via interior point methods \cite{NN94} and 
the symmetric cone programming framework is robust enough to encompass primal dual interior 
point methods \cite{F97,M02,SA03}, which are the pinnacle of the research on interior point methods.

For the special case of linear programming, there are a few well-known algorithms such 
as the simplex method, the ellipsoid method \cite{KH80} and, of course, interior point 
methods. The last two were the only known methods that had polynomial 
time complexity, until a recent work by Chubanov \cite{Ch12,Ch15}. Chubanov's algorithm is 
a new polynomial time algorithm for linear feasibility problems (\emph{i.e.}, $\stdCone = \Re^n_+$)  
with promising computational performance. 
See also the improved version described by Roos \cite{Ro16} and a related work by by Li, Roos and Terlaky \cite{LRT15}.
At this point, it  seems that Chubanov's approach does not fit into previous families of methods.

Our initial motivation for this  work  was an attempt to generalize Chubanov's algorithm to  feasibility problems over symmetric cones, following the work by Kitahara and Tsuchiya \cite{KT16} for second order cone feasibility problems. 
The algorithm in \cite{KT16} works by finding a succession of half-spaces that confine the solutions of a scaled version of \eqref{eq:feas_p}. These half-spaces are found through a ``basic procedure'' and 
then they are adjusted to ensure that the volume of the intersection with $\stdCone$ is sufficiently small. The algorithm progresses by 
keeping track of the volume reductions and it either proves \eqref{eq:feas_p} is feasible/infeasible or proves that a scaled version of \eqref{eq:feas_p}, if feasible, only has solutions that are ``very close'' to the boundary of the cone $\stdCone$, where ``very close'' is adjustable via an 
$\epsilon$ parameter.  The scaled problem in \cite{KT16} makes use of the usual ``infinity-norm'' on $\Re^n$. Here we use what is, perhaps, the proper generalization to the Jordan algebraic setting: 
see the $\norm{\cdot}_{1,\infty}$ spectral norm in Section \ref{sec:norms}.

As in \cite{KT16}, the algorithm described here has some flavor of the ellipsoid method \cite{KH80} since we use volumes to measure progress. It also has some kinship to interior point methods and, in fact, a 
classical bound for self-concordant functions (Lemma \ref{lemma:log}) makes an appearance when we
prove complexity bounds, see also Theorems \ref{theo:step} and \ref{theo:vol2}. In both theorems, we make use of a self-concordant barrier for $\stdCone$ and we explicitly compute a function $\varphi$ that bounds the volume reduction. 
Furthermore, we seek the maximum possible volume reduction that our analysis permits.
These aspects of our discussion seem to be novel and we note that self-concordant functions have not appeared in previous works \cite{Ch15,PS16,KT16}.
We remark that Pe\~{n}a and Soheili \cite{PS16} also describe 
methods for solving \eqref{eq:feas_p} which are partly inspired by Chubanov's work. Later in 
this section, we will contrast their work with ours.

We now present some of the main ideas. Every symmetric cone can be written as a direct product of symmetric cones $\stdCone = \stdCone _1\times \cdots \times \stdCone _\ell$, 
where each $\stdCone _i$ is a \emph{simple} symmetric cone contained in some finite dimensional space $\jAlg _i$. We also have $\jAlg = \jAlg _1 \times \cdots \times \jAlg _\ell$. 
The dimension and rank of each $\jAlg _i$ are denoted by $d_i,r_i$, respectively, and we define $d = d_1 + \cdots + d_\ell$ and $r = r_1 + \cdots + r_\ell$.
See Section \ref{sec:pre} for a review on the necessary notions on symmetric cones. 
Associated to \eqref{eq:feas_p}, we 
have the problem 
	\noindent\begin{align}
	{\mathrm{find}} & \quad y,u \tag{D} \label{eq:feas_d} \\ 
	\mbox{subject to} & \quad y = \stdMap^\T u \notag \\
	& \quad y \in \stdCone, y \neq 0, \notag
	\end{align} 
Due to the Gordan-Stiemke's theorem (see Corollary 2 in Luo, Sturm and Zhang \cite{Luo97dualityresults}), \eqref{eq:feas_p} is feasible if and only if 
\eqref{eq:feas_d} is infeasible.

Since multiplying a solution of \eqref{eq:feas_p} by a positive constant gives 
rise to another solution of \eqref{eq:feas_p}, it makes sense 
to consider a scaled version such as follows:
\begin{align}
{\mathrm{find}} & \quad x \tag{P$_{\text{scaled}}^\stdMap$} \label{eq:feas_p_s} \\ 
\mbox{subject to} & \quad \stdMap x = 0 \notag \\
 & \quad \norm{x}_{1,\infty} \leq \SC \notag \\
& \quad x \in \interior \stdCone, \notag
\end{align} 
where $\norm{x}_{1,\infty}$ denotes the maximum of the ``$1$-spectral norms'' of the blocks of $x$. 
That is, we have $\norm{x}_{1,\infty} = \max\{ \norm{x_1}_{1},\ldots, \norm{x_\ell}_{1}\}$, where $\norm{x_i}_1$ is the sum of the absolute values of the eigenvalues of $x_i$. See Section \ref{sec:norms} for more details.
The usage of $\norm{\cdot}_{1,\infty}$ might seem arbitrary at first, but it is the key for obtaining good complexity bounds in one of the phases of our algorithm.

Now, suppose that $\epsilon > 0$ is given. We say that $x$ is an \emph{$\epsilon$-feasible} solution 
to \eqref{eq:feas_p_s} if $x$ is feasible for \eqref{eq:feas_p_s} and $\lambda _{\min}(x) \geq \epsilon$, where $\lambda _{\min}(x)$ is the minimum eigenvalue of $x$. The algorithm we describe here accomplishes one of the following three goals:
\begin{enumerate}
	\item to find a feasible solution to \eqref{eq:feas_p},
	\item to prove \eqref{eq:feas_p} is infeasible by finding a solution to \eqref{eq:feas_d},
	\item to prove that the minimum eigenvalue of any solution to \eqref{eq:feas_p_s} must be smaller than $\epsilon$.\footnote{ 
	As in the case of semidefinite matrices, an element of $x \in \interior \stdCone$ with small minimum eigenvalue is very close to the boundary of $\stdCone$. So, while item $3$ does not rule out 
	the possibility of \eqref{eq:feas_p} being infeasible, it shows that it is very close to being infeasible. 
}
\end{enumerate}

Denote the set of feasible solutions of \eqref{eq:feas_p_s} by $\feasPs$. Let 
$\hat x = \stdInt/r$, where $\stdInt$, $r$ are the identity element and the rank of $\jAlg$, respectively.
The  approach we describe here is based on the fact that if $\hat x$ is not feasible for \eqref{eq:feas_p}, we can find in a reasonable time through a so-called ``basic procedure'' a 
solution to either \eqref{eq:feas_p} or \eqref{eq:feas_d} or, failing that, for at least 
one of the blocks $\jAlg _k$ we obtain a half-space $H\subseteq \jAlg _k$ with very special properties. 

This half-space can be seen as a certificate that if $x \in \feasPs$, then $x_k$ must be close to the boundary of $\stdCone _k$. Furthermore, we get a concrete bound on $\lambda _{\min}(x_k)$, see Lemma \ref{lemma:stop}. 
This is analogous to the observation that in Chubanov's original algorithm, once the basic procedure finishes without finding solutions to either \eqref{eq:feas_p} or \eqref{eq:feas_d},  there is a least one coordinate $k$ for which we know for sure that $x_k \leq 1/2$ holds for all $x$ that are feasible for a scaled version of \eqref{eq:feas_p}, see the comments after Lemma 2.1 in \cite{Ch15}. 
In fact, when $\stdCone = \Re^n_+$, the constraint $\norm{x}_{1,\infty} \leq \SC $ in \eqref{eq:feas_p_s} forces 
$x$ to belong to the unit cube $[0,1]^n$, which is exactly the same type of scaling used in
\cite{Ch15}. Furthermore, when $\stdCone$ is a direct product of second order cones, if $x \in \stdCone$ then  $\norm{x}_{1,\infty}$ is twice the largest component of 
$x$ and thus become the usual (non-spectral) ``infinity-norm'' times a positive constant, which is the same norm 
used in \cite{KT16} up to a positive constant.

More concretely, after the basic procedure (Algorithm \ref{alg:basic}) ends, we obtain at least one index $k$ together with elements $w_k,v_k \in \jAlg _k$ which define  a half-space  $H(w_k,v_k) = \{x_k \in \jAlg _k \mid \inProd{x_k}{w_k} \leq \inProd{w_k}{v_k}\}$ such that the following three key properties hold.
\begin{enumerate}[label=({P.\arabic*})]
\item If $x \in \feasPs$, then $x_k \in   H(w_k,v_k)$.\label{p:1}
\item There is a linear bijection $Q = Q_1 \oplus \cdots \oplus Q_\ell$ such that $Q(\stdCone) = \stdCone $ and $Q_k (H(\stdInt _k,\stdInt_k /r_k))  = H(w_k,v_k)$.\label{p:2}
\item The volume $H(w_k,v_k) \cap \stdCone _k$ is less than the volume of $H(\stdInt_k,\stdInt _k/r_k )\cap  \stdCone _k$ and their ratio is 
smaller than a positive constant $\delta < 1$.\label{p:3}
\end{enumerate}
Once $Q$ is found, we may consider the problem
\begin{align}
  {\mathrm{find}} & \quad x \tag{P$_{\text{scaled}}^{\stdMap Q}$}  \label{eq:feas_p_s2} \\ 
  \mbox{subject to} & \quad \stdMap Q x = 0 \notag \\
  & \quad \norm{x}_{1,\infty} \leq \SC  \notag \\
  & \quad x \in \interior \stdCone. \notag
 \end{align} 
We will discuss \ref{p:1}, \ref{p:2} and \ref{p:3} in Section \ref{sec:vol}, 
but in a nutshell, we construct $Q$ in such a way that if $x$ is feasible for \eqref{eq:feas_p_s}, then  $Q^{-1}(x)$ is feasible for \eqref{eq:feas_p_s2}. On the other hand, if $x$ is feasible for \eqref{eq:feas_p_s2}, then $Q(x)$ is a feasible solution to the original problem \eqref{eq:feas_p}.

Since the system \eqref{eq:feas_p_s2} has the same shape as \eqref{eq:feas_p}, we may apply the basic procedure again. Then, \ref{p:3} ensures that we are making progress towards confining the $k$-th block of the feasible region of \eqref{eq:feas_p_s} to a region inside of $\stdCone _k$ with smaller volume. As in Chubanov's algorithm, when $\stdCone _k$ is the half-line $\Re_+$, \ref{p:3} ensures that 
the $x_k$ is contained in an interval $[0,\delta]$, with $\delta < 1$.

The main complexity result are given in Propositions \ref{prop:bp} and Theorem 
\ref{theo:main}. Our version of the basic procedure (Algorithm \ref{alg:basic}) takes at most $\ell^3 r_{\max}^2$ iterations, where $r_{\max} = \max \{r_1,\ldots,r_\ell \}$. The main algorithm (Algorithm \ref{alg:main}) takes at most $\frac{r}{\varphi(2)}\log \left(\frac{1}{\epsilon}\right)$ iterations, where $\varphi(2)$ is a positive constant as in Theorems \ref{theo:step} and \ref{theo:vol2}.

Here we highlight some of the keys aspects of our analysis while contrasting with the approach by Pe\~na and Soheili \cite{PS16}.
\begin{enumerate}
	\item Our analysis take into consideration the block division of $\stdCone$. This is expressed, in part, by our usage of the $\norm{\cdot}_{1,\infty}$ spectral norm, which we believe is novel in this context and 
	has advantages even when $\ell = 1$, when it becomes the $\norm{\cdot}_1$ spectral norm. As such, the complexity of the basic procedure (Algorithm \ref{alg:basic}) and the overall complexity is expressed in terms of $\ell$ and $r_{\max} = \max \{r_1, \ldots, r_{\ell}\}$. Our version of the basic procedure 
	performs at most $\mathcal{O}(\ell ^3 r_{\max}^2)$ iterations (Proposition \ref{prop:bp}) and is analogous to the ``Von Neumann scheme'' discussed in \cite{PS16}, which, as the authors remark, when specialized to $\stdCone = \Re_+^n$ gives essentially the same Basic Procedure described by Chubanov. Their 
	version using the  ``Von Neumann scheme'' produces a basic procedure 
	which performs at most $\mathcal{O}(r^4)$ iterations, where $r = r_1 + \cdots + r_\ell$. Although not always, there are at least a few cases of interest where $\ell ^3 r_{\max}^2$ is a better bound than $r^4$. 
	
	Consider, for instance, the case where the rank of the blocks are equal or close to being equal. If we have product of $\ell$ second order cones, we get a complexity of $\mathcal{O}(\ell^3)$ vs $\mathcal{O}(\ell ^4)$. The same bounds hold for the case $\stdCone = \Re^n_+$. In the specific case of semidefinite programming over the $n\times n$ matrices, we get $\mathcal{O}(n^2)$ vs $\mathcal{O}(n^4)$. One case for which the bound in \cite{PS16} is better is when $\ell = 30$, $r_{\max} = 21$ and $r = 50$.
		
	In \cite{PS16}, the authors also suggest the usage of the so-called ``smooth perceptron'' \cite{SP12} which lowers the number of iterations to $\mathcal{O}(r^2)$, although each iteration is more expensive.
	The approach we describe here can easily be adapted to use the smooth perceptron, thus yielding a Basic Procedure requiring no more than $\mathcal{O}(\ell ^{3/2} r_{\max})$ iterations, which still has better performance in the cases we mentioned, see the remarks after Proposition \ref{prop:bp_c}.

	\item The progress of the algorithm in \cite{PS16} is measured through the quantity $\chi = \max \{\det x \mid  \stdMap x = 0, x \in \interior \stdCone, \norm{x}^2 = r \}$. 
	The performance of the approach in \cite{PS16} also depends on $\chi$ and this leads to a very elegant complexity analysis.
	Nevertheless, we think some of the geometric intuition is lost by using $\chi$.
	We hope to recover that by painting an intuitive picture in terms of volumes and half-spaces.
	
	Furthermore, the algorithm in \cite{PS16} runs until a solution to \eqref{eq:feas_d} or 
	\eqref{eq:feas_p} is found, which happens in a finite number of iterations when $\eqref{eq:feas_p}$ is feasible. We remark that it is possible that the algorithm in 
	\cite{PS16} loops infinitely and this is also true for Algorithm 2 of \cite{PS16}, see the comments after Thereom 1 therein. 
	Clearly, in a practical implementation, we also have to introduce some stopping criteria that depends, perhaps, on the machine precision, since it is meaningless to continue computations after the intermediate numbers get smaller than a certain threshold.
	
	In this sense, it is of great interest to know what could be said about the eigenvalues of	the solutions of \eqref{eq:feas_p} or of some scaled version of it, when the algorithm in \cite{PS16} is stopped prematurely. 
	In contrast, at each iteration of Algorithm \ref{alg:main} we get concrete upper bounds on $\lambda _{\min}(x)$, for $x \in \feasPs$, see Lemma \ref{lemma:stop}. This is possible, in part, due to the volumetric considerations in Proposition \ref{prop:min_vol}. We believe a similar discussion would have been useful in \cite{PS16}. We should remark however, that Theorem 3 in \cite{PS16} implies that after $k$ iterations of their Algorithm 4, the bound 
	$\chi \leq 1.5^{-k}$ holds.

	\item At the end of the basic procedure in \cite{PS16}, a vector $z$ is obtained and only the component with largest eigenvalue takes a role in the actual rescaling of the problem, see Step 4 in Algorithm 4 in \cite{PS16}. Furthermore, a fixed step size is used throughout the algorithm (the parameter $a$ in Algorithm 4 in \cite{PS16}).  
	
	In our approach, after the basic procedure ends, we obtain a vector $y$ but do not discard the smaller eigenvalues.  We then check for each block $\stdCone_k$, whether $y_k$ affords a volume reduction for that block and if it does, we take the maximum possible step size our analysis permits, see Theorems \ref{theo:step} and \ref{theo:vol2}.
	This corresponds to Lines \ref{alg:main:rho} and \ref{alg:main:if} in Algorithm \ref{alg:main} and this behavior also sets our approach apart from the algorithm in \cite{KT16}. A graph relating the step size $\rho$ and the volume reduction can be seen in Figure \ref{fig:vol}.
\end{enumerate}

This article is divided as follows. On Section \ref{sec:pre} we review some notions related to symmetric cones. On Section \ref{sec:vol} we discuss several notions regarding volumes in symmetric cones, in particular, we discuss how to confine the blocks of \eqref{eq:feas_p_s} to a region with smaller volume starting from an appropriate $y$ vector. On Section \ref{sec:bp}, we discuss the basic procedure and on Section \ref{sec:full} we discuss the main algorithm. Section \ref{sec:conc} concludes this paper.
In Appendix \ref{app:sdp}, for ease of reference, we restate Algorithms \ref{alg:basic} and \ref{alg:main} for the case of semidefinite programming.

\section{Preliminary considerations}\label{sec:pre}
Here, we review some  aspects of the theory of symmetric cones and Euclidean Jordan algebras. 
 More 
details can be found in the book by Faraut and Kor\'anyi \cite{FK94} and also 
in the survey article by Faybusovich \cite{FB08}.  Let $\jAlg$ be a finite dimensional Euclidean space equipped with an inner product $\inProd{\cdot}{\cdot}$, and $\jProd{}{}:\jAlg\times \jAlg \to \jAlg$ be a bilinear 
map. We say that  $(\jAlg, \jProd{}{})$ is an Euclidean Jordan algebra if the following properties are satisfied.
\begin{enumerate}
	\item $\jProd{y}{z} = \jProd{z}{y}$,
	\item $\jProd{y}({\jProd{y^2}{z}}) = \jProd{y^2}({\jProd{y}{z}})$, where $y^2 = \jProd{y}{y}$,
	\item $\inProd{\jProd{y}{z}}{w} = \inProd{y}{\jProd{z}{w}}$,
\end{enumerate}
for all $y,w,z \in \jAlg$. All symmetric cones arise as the cone of squares of some Euclidean Jordan algebra.
That is, for a symmetric cone $\stdCone$, we have $\stdCone = \{\jProd{x}{x} \mid x \in \jAlg \}$, for some Euclidean Jordan algebra $(\jAlg,\jProd{}{})$. See Theorems III.2.1 and III.3.1 in \cite{FK94}, for more details.
Furthermore, 
we can assume that $\jAlg$ has an (unique) identity element $\stdInt$ satisfying $\jProd{y}{\stdInt} = y$, for all $y \in \jAlg$.

In what follows, we say that $c$ is \emph{idempotent} if 
$\jProd{c}{c} = c$. Moreover, $c$ is \emph{primitive} if it is nonzero and there is no way of writing 
$c = a+b$ with $a$ and $b$ nonzero idempotent elements satisfying $\jProd{a}{b} = 0$.

\begin{theorem} [Spectral theorem, see Theorem III.1.2 in \cite{FK94}]\label{theo:spec}
	Let $(\jAlg, \jProd{}{} )$ be an Euclidean Jordan Algebra and let $x \in \jAlg$. Then there are:
	\begin{enumerate}[label=({\it \roman*})]
		\item primitive idempotents $c_1, \dots, c_r$ satisfying
	\begin{flalign}
	& \jProd{c_i}{c_j} = 0 \, \,\,\,\, \qquad \qquad \text{ for } i \neq j, \\
	&\jProd{c_i}{c_i} = c_i, \, \, \, \, \qquad \qquad  i = 1, \ldots, r, \\
	& c_1 + \cdots + c_r  = \stdInt, \qquad i = 1, \ldots, r,
	\end{flalign}
		\item unique real numbers $\eig _1, \ldots, \eig _r$ satisfying
		\begin{equation}		
		x = \sum _{i=1}^r \eig _i c_i \label{eq:dec}.
		\end{equation}
	\end{enumerate}
Moreover, $r$ only depends on $\jAlg$.
\end{theorem}
We say that the $c_1, \ldots , c_r$ in Theorem \ref{theo:spec} form a  \emph{Jordan Frame} for $x$. The 
$\lambda _1, \ldots, \lambda _r$ are the \emph{eigenvalues} of $x$. We will denote the maximum eigenvalue of 
$x$ by $\eig  _{\max} (x)$ and the minimum by $\eig _{\min} (x)$.
The $r$ appearing in the theorem is 
called the \emph{rank} of $\jAlg$ and depends only on the algebra $\jAlg$.  The rank of $x$ is the number of nonzero eigenvalues.
Given $x \in \jAlg$, we define its trace by
$$
\tr x = \eig _1 + \cdots + \eig _r,
$$
where $\eig _1, \cdots, \eig _r$ are the eigenvalues of 
$x$. The trace function is linear and can be used to define an 
inner product for $\jAlg$. So, henceforth, we will assume that the inner product $\inProd{\cdot}{\cdot}$ is defined as follows.
$$
\inProd{x}{y} = \tr (\jProd{x}{y}).
$$
When $\jAlg$ is the space of $n\times n$ symmetric matrices, this inner product becomes the usual Frobenius inner product.
Note that if $c$ is a primitive idempotent, we have $\inProd{c}{c} = 1$. Furthermore, we have $\inProd{\stdInt}{\stdInt} = r$.
We also have a determinant function defined as 
$$
\det x = \eig_1 \times \cdots \times \eig _r.
$$

If the rank of $x$ is equal to $r$, it admits an inverse element denoted by $x^{-1}$. We 
have $x^{-1} = \sum _{i=1}^r \eig _i^{-1} c_i $. We then have $\jProd{x}{x^{-1}} = \stdInt$. 

We recall the following properties of $\stdCone$. The results follows 
from various propositions that appear in 
\cite{FK94}, such as Proposition~III.2.2 and Exercise 3 in Chapter III. See also Equation (10) in 
\cite{Sturm2000}.
\begin{proposition}\label{prop:aux}
	Let $x \in \jAlg$.
\begin{enumerate}[label=({\it \roman*})]
		\item $x \in \stdCone$ if and only if the eigenvalues of $x$ are nonnegative. \label{paux:1}
		\item $x \in \interior \stdCone$ if and only if the eigenvalues of $x$ are positive. \label{paux:2}
\end{enumerate}
\end{proposition}
It follows that if $x \in \stdCone$, it admits a square root defined as $\sqrt{x} = \sum _{i=1}^r \sqrt{\eig} _i c_i$. With that, we have $\jProd{\sqrt{x}}{\sqrt{x}} = x$. We will also write 
$x^{1/2}, x^{-1/2}$ for $\sqrt{x}$ and $\sqrt{x^{-1}}$, respectively.

An Euclidean Jordan Algebra $(\jAlg, \jProd{}{})$ is said to be \emph{simple} if 
it is not possible to decompose it as an orthogonal direct sum $\jAlg = \jAlg _1 \oplus \jAlg _2$ with $\jAlg _1$ and $\jAlg _2$ 
being themselves nonzero Euclidean Jordan algebras. When $\stdCone$ is the cone of squares of a 
simple Euclidean Jordan algebras, we will also say that it is a \emph{simple cone}. 
In this paper, we assume the following decomposition
\begin{align*}
\jAlg & = \jAlg _1 \oplus \ldots \oplus \jAlg _\ell\\
\stdCone & = \stdCone _1 \oplus \ldots \oplus \stdCone _\ell,
\end{align*}
where the $\jAlg _i$ are simple Euclidean Jordan Algebras of rank $r_i$ and 
$\stdCone _i$ is the cone of squares of $\jAlg _i$. The dimension of each $\jAlg _i$ is 
denoted by $d_i$ and we have $d = d_1 + \cdots + d_\ell$ and $r = r_1 + \cdots + r_\ell$.
We also have $\stdInt = (\stdInt _1, \ldots,\stdInt _\ell)$, where $\stdInt _i$ is 
the identity in $\jAlg _i$.
Note that orthogonality 
expressed by this decomposition is not only with respect the inner product $\inProd{\cdot}{\cdot}$ 
but also with respect the Jordan Product $\jProd{}{}$, meaning that $\jProd{x_i}{y_j} = 0$ if 
$x_i \in \jAlg_i, y_j \in \jAlg_j$ with $i \neq j$.

If $x \in \jAlg$, we  write $x_i$ for the corresponding $i$-th block of $x$. That is, we have $x = (x_1,\ldots, x_\ell)$, where each $x_i \in \jAlg_i$. 
With this decomposition, the theory described so far can be applied in a blockwise fashion. So, we have, for instance, 
$\jProd{x}{y} = (\jProd{x_1}{y_1},\ldots,\jProd{x_\ell}{y_\ell})$,
$\tr x = \sum _{i=1}^\ell \tr x_i$, $\det x = \prod_{i=1}^{\ell} \det x_i$ and
$\lambda _{\min}(x) = \min \{\lambda _{\min}(x_1), \ldots, \lambda _{\min}(x_{\ell}) \}$.

We also have a so-called \emph{quadratic representation} $Q_x$ defined as the linear 
map such that 
$$
\Qr{x}(y) = 2\jProd{x}{(\jProd{x}{y})} - \jProd{x^2}{y}.
$$
We also have $\Qr{x}(y)  = (\Qr{x_1}(y_1), \ldots, \Qr{x_\ell}(y_\ell)  )$.
Denote by $I_k$ the identity map on $\jAlg _k$, we have $\Qr{\stdInt_k} = I_k$.
The quadratic representation has the following key properties, see Section 1.2 in \cite{Sturm2000} and Section \RNum{2}.3 together with  Proposition \RNum{3}$.4.2$ in 
\cite{FK94}.
\begin{proposition}\label{prop:quad}
Suppose $x \in \interior \stdCone$
\begin{enumerate}[label=({\it \roman*})]
	\item $\Qr{x}$ is a self-adjoint linear map on $\jAlg$ such that $\Qr{x}(\stdCone) = \stdCone$.
	\item $\Qr{x}^{-1} = \Qr{x^{-1}}$,  $\Qr{x}(\stdInt) = x^2$, $\Qr{x}(x^{-1}) = x$, $\Qr{ax} = a^2 \Qr{x}$ for $a \in \Re$.
	\item if $\stdCone$ is simple then $\det \Qr{x} = (\det x)^{\frac{2d}{r}}$.
\end{enumerate}
\end{proposition}
We remark that since $\Qr{x}$ is a self-adjoint linear map over the linear space 
$\jAlg$, all its eigenvalues are real. In item $(iii)$, the determinant of $\Qr{x}$ is, of course, the product of the eigenvalues of $\Qr{x}$. Another way of visualizing the determinant is recalling 
that after identifying $\jAlg$ with some $\Re^d$, $\Qr{x}$ is representable as a square matrix $S$, so the determinant that appears in item $(iii)$ refers to the determinant of that $S$.
See Section 1.2 in the work by Sturm \cite{Sturm2000} for the relations 
between eigenvalues of $\Qr{x}$ and $x$.

\subsection{Spectral norms for general Jordan algebras}\label{sec:norms}

In this paper, a few different norms will come into play\footnote{It is not entirely obvious that $\norm{\cdot}_{\infty}$ and $\norm{\cdot}_1$ are norms in the 
	general setting of Jordan algebras, for more details see Example 2 in \cite{FB08}.}. We will write $\norm{\cdot}$ or  $\norm{\cdot}_2$ for the norm 
induced by $\inProd{\cdot}{\cdot}$. We have $\norm{x} = \sqrt{\tr(x^2)} = \sqrt{\lambda_1^2 + \cdots + \lambda _r^2}$.
We will also define $\norm{x}_1 = \abs{\lambda _1} + \cdots + \abs{\lambda _r}$ and 
$\norm{x}_{\infty} = \max\{\abs{\lambda _1}, \ldots, \abs{\lambda _r}\} = \max \{\lambda _{\max}(x),-\lambda_{\min}(x) \}$. Note that if 
$x \in \stdCone$, we have $\norm{x}_1 = \inProd{x}{\stdInt}$ and $\norm{x}_{\infty} = \lambda _{\max}(x)$. 

Recall that $\jAlg$ is not necessarily a simple algebra, but  admits a decomposition $\jAlg = \jAlg _1 \oplus \ldots \oplus \jAlg _\ell$, where each $\jAlg _i$ is indeed simple. 
In this work, we will 
consider norms that take into account the way that $\jAlg$ is decomposed as simple algebras. 
We will use $\norm{x}_{1,\infty}$ to denote the maximum among the $1$-norms of the blocks. That is 
$$
\norm{x}_{1,\infty} = \max \{\norm{x_1}_1, \ldots,\norm{x_\ell}_1  \}.
$$
So that when $\ell = 1$, we have $\norm{x}_{1,\infty} = \norm{x}_1$. Furthermore, when $ x \in \stdCone$, 
we have 
$$\norm{x}_{1,\infty} = \max \{\inProd{x_1}{\stdInt_1}, \ldots,\inProd{x_\ell}{\stdInt_\ell} \}.$$
We also define $\norm{x}_{\infty,1} = \norm{x_1}_{\infty} + \cdots + \norm{x_\ell}_{\infty}$.

Recall that if $g(\cdot)$ is an arbitrary norm on $\jAlg$, we define its conjugate norm 
by $g(x)^* = \sup \{\inProd{x}{y} \mid g(y) \leq 1 \}$. With that, 
we have the generalized Cauchy-Schwarz inequality $\inProd{x}{y}\leq g(x)g(y)^*$.
We then have the following 
result, which is unsurprising but requires a careful proof.
\begin{proposition}\label{prop:inequality}
Let $y \in \jAlg$. Then
\begin{enumerate}[label=({\it \roman*})]
	\item $\norm{y}_{1}^* = \norm{y}_{\infty}$,
	\item $\norm{y}_{1,\infty}^* = \norm{y}_{\infty,1}$,
	\item $\norm{y}_{\infty,1} \leq \sqrt{\ell}\norm{y}_2$.
\end{enumerate}
\end{proposition}
\begin{proof}
\begin{enumerate}[label=({\it \roman*})]
	\item By definition, we have
	$$\norm{y}_{1}^* = \sup \left\{  \inProd{y}{z} \mid \norm{z}_{1} \leq 1 \right\}.$$
	Let $z \in \jAlg$ write its Jordan decomposition as 
	$z = \sum _{i=1}^r \lambda _i c_i$. Note that condition $\norm{z}_{1} \leq 1$  means 
	that $\abs{\lambda _1} + \cdots + \abs{\lambda _r}  \leq 1$. Let 
	$\mathcal{I}$ denote the set of primitive idempotents in $\jAlg$. We have
	\begin{align*}
	\norm{y}_{1}^* & =\sup _{c_1,\ldots, c_r \in \mathcal{I}}  \sup _{\abs{\lambda _1} + \cdots + \abs{\lambda _r}  \leq 1} \sum _{i=1}^r \lambda _i \inProd{y}{c_i} \\
	& = \sup _{c_1,\ldots, c_r \in \mathcal{I}} \max\{ \abs{{\inProd{y}{c_1}}},\ldots,\abs{\inProd{y}{c_r}} \}\\
	& = \sup _{c \in \mathcal{I}} \abs{\inProd{y}{c}},
	\end{align*}
	where the second equality follows from the fact that $\sum _{i=1}^r \lambda _i \inProd{y}{c_i} $ is 
	the Euclidean inner product between the vector $(\lambda _1, \ldots, \lambda _r)$ and 
	$(\inProd{y}{c_1},\ldots, \inProd{y}{c_r})$ in the $\Re^r$ space. In $\Re^r$, we 
	know that the usual $1$-norm and the usual $\infty$-norm are conjugated pairs and this is 
	why the equality follows. 
	
	Now, $\sup _{c \in \mathcal{I}} \abs{\inProd{y}{c}} = \sup _{c \in \mathcal{I}} \max\{\inProd{y}{c}, -\inProd{y}{c} \}$ holds. 
	Then, a result from Hirzebruch implies that $\sup _{c \in \mathcal{I}} \inProd{y}{c} = \lambda _{\max}(y)$ and $ \sup _{c \in \mathcal{I}} \inProd{y}{c} = - \lambda _{\min}(y)$.
	See either \emph{Satz} 2.3 in \cite{H70},  Exercise 4 in Chapter 4 of \cite{FK94} or Equation $(9)$ in \cite{Sturm2000}. This shows that $\norm{y}_{1}^* = \norm{y}_\infty$.

	\item By definition, we have
	$$\norm{y}_{1,\infty}^* = \sup \left\{ \sum _{i=1}^\ell \inProd{y_i}{z_i} \mid \norm{z}_{1,\infty} \leq 1 \right\}.$$
	Note that if $\norm{z}_{1,\infty} \leq 1$, then  $\norm{z_i}_1 \leq 1$, for all $i$.
	By the generalized Cauchy-Schwarz inequality, we have $\inProd{y_i}{z_i} \leq \norm{y_i}_\infty \norm{z_i}_1 \leq  \norm{y_i}_\infty$. Therefore, $\norm{y}_{1,\infty}^* \leq \norm{y}_{\infty,1}$. To show that 
	$\norm{y}_{1,\infty}^* = \norm{y}_{\infty,1}$, it is enough to construct $z$ such that 
	$\inProd{y}{z} = \norm{y}_{\infty,1}$ and $\norm{z}_{1,\infty} \leq 1$. 
	We construct $z$ in a block-wise fashion. First, let $y_i = \sum _{j=1}^{r_i} {\lambda _{j}} c_j$ be the Jordan decomposition of $y_i \in \jAlg _{i}$, where $\lambda _1 \leq \cdots \leq \lambda _{r_i}$. Then 
	$\norm{y_i}_{\infty} = \max \{-\lambda _1, \lambda _{r_i} \} $. We then let $z_i$ be  either 
	$c_1,-c_1$ or $c_{r_i}$, so that $\inProd{y_i}{z_i} = \norm{y_i}_{\infty}$ and $\norm{z_i}_1 = 1$. If we construct 
	the blocks of $z$ in this way we have $\inProd{y}{z} = \norm{y}_{\infty,1}$ and  $\norm{z}_{1,\infty} \leq 1$.
	\item Consider the vector $v = (\norm{y_1}_{\infty}, \ldots, \norm{y_\ell}_{\infty}) \in \Re^\ell$. Due to the Cauchy-Schwarz inequality we have
	\begin{align*}
	\norm{y}_{\infty,1} &= \inProd{v}{(1,\ldots,1)}  \\	
	& \leq \sqrt{\ell}\sqrt{\norm{y_1}_{\infty}^2 +  \cdots + \norm{y_\ell}_{\infty}^2} \\
	& \leq \sqrt{\ell} \sqrt{\norm{y_1}_{2}^2 +  \cdots + \norm{y_\ell}_{2}^2} \\
	&  = \sqrt{\ell} \norm{y}_2.
	\end{align*}
\end{enumerate}	
\end{proof}
We have the following corollary.
\begin{corollary}
Let $x,y \in \jAlg$. The following hold:
\begin{align}
\inProd{y}{x} &\leq \norm{y}_{1}\norm{x}_{\infty} \label{eq:inf}\\
\inProd{y}{x} &\leq \norm{y}_{1,\infty}\norm{x}_{\infty,1} \label{eq:1_inf}
\end{align}
\end{corollary}

\subsection{Examples}\label{sec:examples}
There is a classification of the finite simple symmetric cones and here we briefly review some of the more widely used 
examples. When $\jAlg = \S^n$ is the space of $n\times n$ symmetric matrices we can 
define $\jProd{x}{y}$ as $\frac{xy + yx}{2}$. With that, $\stdCone$ is the (simple) cone $\PSDcone{n}$ of $n\times n$ 
positive semidefinite matrices, which has rank $n$. We have that $\norm{\cdot}$ becomes the Frobenius norm. In this case, 
the eigenvalues as discussed in Theorem \ref{theo:spec} are exactly the same 
eigenvalues that appear in classical linear algebra. The primitive idempotents correspond to rank $1$ matrices with norm $1$.
So Theorem \ref{theo:spec} expresses the fact that any symmetric matrix can be written as linear combination of mutually orthogonal rank $1$ matrices.
Finally, the quadratic map $\Qr{x}$ is such that $\Qr{x}(y) = xyx$. For $x \in \stdCone$, we have $\norm{x}_1 = \tr x$ and $\norm{x}_{\infty} = \lambda _{\max}(x)$, where $\tr$ is the usual matrix trace.

Another interesting case is when $\jAlg = \Re^{n} = \Re\times \Re^{n-1}$ and $\Re^{n-1}$ is equipped the usual Euclidean inner product.
Let $x \in \jAlg$, we decompose it as $x = (x_1,\overline{x})$, where $x_1 \in \Re$ and $\overline{x} \in \Re^{n-1}$. 
We then define $\jProd{x}{y} = (x_1y_1 + \inProd{\overline{x}}{\overline{y}},x_1\overline{y} + y_1\overline{x})$. With that, $\stdCone$ becomes 
the Lorentz cone $\SOC{n} = \{x \in \jAlg \mid x_1^2 \geq x_2^2 + \ldots + x_n^2, x_1 \geq 0\}$. Note that the 
identity element is $\stdInt = (1,0,\ldots, 0)$. Given  $x \in \jAlg$, its eigenvalues and corresponding idempotents are
\begin{align}
\lambda _1 = x_1 + \norm{\overline{x}} & \qquad \lambda _2 = x_1 - \norm{\overline{x}} \label{eq:eig_soc}\\
c_1  = 
\begin{cases}
\frac{1}{2}(1,z)  & \quad \norm{\overline{x}} = 0\\
\frac{1}{2}\left(1,\frac{\overline{x}}{\norm{\overline{x}}}\right)  &\quad \norm{\overline{x}} \neq 0 
\end{cases} & \qquad c_2  = 
\begin{cases}
\frac{1}{2}(1,-z)  & \quad \norm{\overline{x}} = 0\\
\frac{1}{2}\left(1,-\frac{\overline{x}}{\norm{\overline{x}}}\right)  & \quad \norm{\overline{x}} \neq 0, 
\end{cases}\label{eq:eig_soc2}
\end{align}
where $z \in \Re^{n-1}$ is any vector satisfying $\norm{z} = 1$. See Example 1 in \cite{FB08}, Section 2 in \cite{TT99}, Section 2 and Proposition 1 in \cite{MT00} or  Equations (23) and (24) in \cite{AG03} for more details.  
We have $\det x = x_1^2 - \norm{\overline{x}}^2$ and $\tr (x) = 2x_1 $. Note 
that the inner product $\inProd{x}{y} = \tr(\jProd{x}{y})$ is twice the inner product 
on $\Re^n$, that is, $\inProd{x}{y} = 2(x_1y_1 + \cdots + x_ny_n)$. Note that 
if $x \in \stdCone$, then $\norm{x}_{1} = 2x_{1}$ and $\norm{x}_{\infty} = x_1 + \norm{\overline{x}}$.

The quadratic representation is given by 
$$
\Qr{x} = \begin{pmatrix}
\norm{x}^2 & & & \quad 2x_1\overline{x}^\T \\
2x_1\overline{x} & & & \quad \det xI_{n-1} + 2 \overline{x}\overline{x}^\T
\end{pmatrix},
$$
where $I_{n-1}$ is the $(n-1) \times (n-1)$ identity matrix and $ \overline{x}^\T$ indicates the transpose of the column 
vector $\overline{x}$. The rank of $\SOC{n}$ is $2$. 

In second order cone programming, it is common to consider a direct product 
of $\ell$ Lorentz cones in which case the rank of $\stdCone$ is $2\ell$. If 
$x \in \jAlg$, we have $x = (x_1,\ldots, x_\ell)$ and every $x_i$ can be written 
as $x_{i} = (x_{i1},x_{i2})$ with $x_{i1} \in \Re$, $x_{i2} \in \Re^{d_i-1}$.
In this case, if $x \in \stdCone$, then $\norm{x}_{1,\infty} = \max \, \{2x_{11},\ldots, 2x_{\ell1} \}$.
The condition $x_i \in \stdCone _i$ forces, in particular, $x_{i1}$ to be at least as large as 
the other components of $x_i$.
This means that $\norm{x}_{1,\infty}$ is \emph{twice} the coordinate of $x$ with largest value, and, therefore, is twice the usual (non-spectral) infinity norm, for elements inside the cone $\stdCone$. In particular, for the case of second order cones, the scaled problem \eqref{eq:feas_p_s} is 
essentially equivalent to the one considered in \cite{KT16}.

Finally, we remark that $\stdCone = \Re^n_+$ is not a simple cone. The Jordan-algebraic way of analyzing it is as a
direct product $\Re_+ \times \ldots \times \Re_+$. Note that  $\Re_+$ is  as simple cone of rank $1$. If we take $\jAlg = \Re$ and 
define $\jProd{x}{y} = (x_1y_1,\ldots, x_ny_n)$, the corresponding cone of squares is $\Re_+^n$. We have 
$\tr(x) = x_1+\cdots + x_n$, $\det x = x_1\times\cdots \times x_n$ and $\stdInt = (1,\ldots,1)$. The quadratic representation 
is given by 
$$
\Qr{x} = \begin{pmatrix}
x_1^2 & 0 & 0 \\
0 & \ddots & 0 \\
0 & 0 & x_n^2
\end{pmatrix}.
$$
In this case, the eigenvalues of $x$ are its components, so that $\norm{x}_{\infty} = \norm{x}_{1,\infty} = \max\, \abs{x_i}$ and $\norm{x}_{1} = \norm{x}_{\infty,1} = \abs{x_1} + \cdots + \abs{x_n}$.

\subsection{Remarks on notation} 
For a matrix $\stdMap$, we will write $\ker \stdMap$ for its kernel and $\stdMap^*$ for its adjoint.
If $x \in \jAlg$, we write $x_k$ for the corresponding block belonging to $\jAlg _k$. 
We will also write $w_k \in \jAlg _k$  for the elements of $\jAlg _k$, without necessarily 
assuming that $w_k$ is just a block of some element $w \in \jAlg$. This will appear in Theorem \ref{theo:vol2} and in Algorithm \ref{alg:main}, in order to emphasize the set to which the elements belong.

With this notation in mind, for $w_k,v_k \in \jAlg_k$, we define the 
half-space $H(w_k,v_k) \coloneqq \{x_k \in \jAlg _k \mid \inProd{x_k}{w_k} \leq \inProd{w_k}{v_k} \}$.
We also write $\Vol(w_k,v_k)$ for the volume of the region $H(w_k,v_k) \cap \stdCone_k$, which 
corresponds to the integral $\int _{H(w_k,v_k) \cap \stdCone_k} 1$.

As we are using subindexes to denote the blocks, we will use superindexes to indicate the iteration number in Algorithms \ref{alg:basic} and \ref{alg:main}, which should 
not be confused with exponentiation.

\section{Volumetric considerations}\label{sec:vol}
Let $r_{\max} = \max \{r_1,\ldots, r_\ell \}$.
The algorithm we discuss in this work depends on having at each iteration 
access to an element $y \in \stdCone$ satisfying 
$$
\norm{P_\stdMap  y} \leq \frac{1}{2r_{\max} \sqrt{\ell} } \norm{y}_{{1,\infty}},
$$
where $P_\stdMap$ is the orthogonal projection on the kernel of $\stdMap$. In particular, $P_\stdMap  x = x$ if 
$\stdMap x = 0$ and $P _{\stdMap}x = 0$ if and only if there is $u$ such that $x = \stdMap^{\T}u$. We will, in fact, conduct a more general analysis and suppose that $y$ satisfies
$$
\norm{P_\stdMap  y} \leq \frac{1}{\rho r_{\max}\sqrt{\ell}} \norm{y}_{{1,\infty}},
$$ for some $\rho > 1$.  We are now going to show how to use $y$ to construct $w_k, v_k$ such that $H(w_k,v_k)$ satisfies the properties \ref{p:1}, \ref{p:2} and \ref{p:3} described in Section \ref{sec:int}. We divide our discussion in two parts. 
In Section \ref{sec:vol:simple}, we deal with the case 
where $\stdCone$ is a simple symmetric cone. Then, we remove the simplicity 
assumption and do the general case in Section \ref{sec:vol:general}. 

\subsection{The simple case}\label{sec:vol:simple}
Here, we suppose that $\stdCone$ is a simple symmetric cone, so that the dimension of 
$\jAlg$ is $d$, $\ell = 1$,
$r_{\max} = r$ and $\norm{\cdot}_{1,\infty} = \norm{\cdot}_1$. First, we study a few properties 
of the intersection $H(w,v)\cap \stdCone$, where $H(w,v) = \{x \in \jAlg \mid \inProd{w}{x} \leq \inProd{w}{v} \}$. We 
denote the volume of $H(w,v)\cap \stdCone$ by $\Vol(w,v)$. We then have the following results, 
which are analogous to Proposition 2.2 in \cite{KT16}.

\begin{proposition} \label{prop:vol}
Suppose $w \in \interior \stdCone$. Then,
\begin{align}
Q_{{ w^{-1/2}\sqrt{\inProd{w}{v}}}}(H(\stdInt,\stdInt/r )) & = H(w,v) \label{prop:vol:1}\\
\Vol(w,v) & = \left(\frac{\inProd{w}{v}}{ \sqrt[r]{\det w}}\right)^{d } \Vol(e,e/r) \label{prop:vol:2}
\end{align}
\end{proposition}
\begin{proof}
For the first equation, suppose that $x \in H(\stdInt,\stdInt/r)$. Then, by Proposition \ref{prop:quad}, we have 
$$\inProd{w}{Q_{{ w^{-1/2}\sqrt{\inProd{w}{v}}}}(x)} = \inProd{w}{v}\inProd{Q_{{ w^{-1/2}}}(w)}{x} = {\inProd{w}{v}\inProd{e}{x}} \leq \inProd{w}{v},$$
which shows that $Q_{{ w^{-1/2}\sqrt{\inProd{w}{v}}}} \in H(w,v) $.

Also by Proposition \ref{prop:quad}, we have $Q_{{ w^{-1/2}\sqrt{\inProd{w}{v}}}} ^{-1} = 
\frac{1}{{\inProd{w}{v}}} Q_{{ w^{1/2}}}$. So, suppose that $x \in H(w,v)$.
$$\frac{1}{{\inProd{w}{v}}} \inProd{\stdInt}{\Qr{w^{1/2}}(x)} = \frac{1}{{\inProd{w}{v}}} \inProd{\Qr{w^{1/2}}(\stdInt)}{x} = \frac{1}{{\inProd{w}{v}}} \inProd{w}{x} \leq 1,
$$
which shows that $\frac{1}{{\inProd{w}{v}}} Q_{{ w^{1/2}}}(x) \in H(\stdInt,\stdInt/r)$.

For the second equation, 

\begin{align*}
\int _{H(w,v)\cap \stdCone}1  & = \int _{Q_{{ w^{-1/2}\sqrt{\inProd{w}{v}}}}(H(\stdInt,\stdInt/r)\cap \stdCone)} 1 \\
& = \int _{H(\stdInt,\stdInt/r)\cap \stdCone } |\det Q_{{ w^{-1/2}\sqrt{\inProd{w}{v}}}}| \\
& = \det Q_{{ w^{-1/2}\sqrt{\inProd{w}{v}}}} \Vol(\stdInt,\stdInt/r)
\end{align*}
By  Proposition \ref{prop:quad}, we have 
\begin{align*}
\det  Q_{{ w^{-1/2}\sqrt{\inProd{w}{v}}}}  & = \det \inProd{w}{v} Q_{{ w^{-1/2}}} \\
& =  \inProd{w}{v}^d \det  Q_{{ w^{-1/2}}}\\
& = \inProd{w}{v}^d  (\det w^{-1/2})^{\frac{2d}{r}} \\
& = \left(\frac{\inProd{w}{v}}{ \sqrt[r]{\det w}}\right)^{d }.
\end{align*}
\end{proof}

The next lemma shows that we can use $y$ to construct a hyperplane $H(w,v)$ which contains 
$\feasPs$ and has its roots in the work by Chubanov on LPs, see page 692 of \cite{Ch15}. See also, Lemma 3.1 in \cite{KT16} for the SOCP case.
\begin{lemma}\label{lemma:y_proj}
	Suppose $x$ is feasible for \eqref{eq:feas_p_s} and that $y \in \stdCone$ satisfies
	$$
	\norm{P_\stdMap  y} \leq \const \norm{y}_{{1}},
	$$
	then $x \in H(y,\stdInt/\constInv)$.
\end{lemma}
\begin{proof}
Since $y \in \stdCone$, we have $\norm{y}_{1} = \inProd{y}{\stdInt}$. In addition, 
since $x$ is feasible for \eqref{eq:feas_p_s} and $\ell = 1$, we have $\norm{x}_{1,\infty} = \norm{x}_1 \leq 1$.
Then, by 
using the  Proposition \ref{prop:inequality} and the inequality in \eqref{eq:inf}, we have
	\begin{align*}
	\inProd{y}{ x} = \inProd{y}{ P_\stdMap  x} =  \inProd{P_\stdMap y}{   x} \leq \norm{P_\stdMap  y}_{\infty} \norm{x}_{1} \leq  \norm{P_\stdMap  y}  \leq 
	\const \inProd{y}{\stdInt}. 
	\end{align*}
\end{proof}
Due to Proposition \ref{prop:vol}, as long as $y \in \interior \stdCone$, there is some $Q$ that maps $H(\stdInt,\stdInt /r)$ bijectively into $H(y,\stdInt/ \constInv)$. 
The only problem is that the volume of $H(y,\stdInt/\constInv)\cap \stdCone$ might not be sufficiently small. We will now address this issue and try to give  some intuition on the choices 
we will make ahead.

As long as $\feasPs$ is contained in $H(w,v)\cap \stdCone$, we may select any $w$ and $v$ such that 
the volume of $H(w,v)\cap \stdCone$ is small. However, finding $w$ and $v$ directly is cumbersome. 
Instead, 
we use $y$ and settle for the easier goal of finding $w,v$ such that $H(y,\stdInt /\constInv )\cap H(\stdInt, \stdInt/r)\subseteq H(w,v)$, since this is enough to ensure that $\feasPs$ is contained in $H(w,v)$. While doing so, we seek to minimize $\Vol(w,v)$. 
Theorem 22.3 of \cite{rockafellar}, tells us 
that $H(y,\stdInt /\constInv )\cap H(\stdInt, \stdInt/r)\subseteq H(w,v)$ if and only if 
there are $\alpha \geq 0$, $\beta \geq 0$ such that 
\begin{align}
\alpha y + \beta \stdInt & = w \label{eq:hyper1}\\
\alpha  \frac{\inProd{y}{\stdInt}}{\constInv}  + \beta & \leq\inProd{v}{w} \label{eq:hyper2}.
\end{align}
Furthermore, for fixed $v \in \interior \stdCone$, the $w$ that minimizes $\Vol(w,v)$ is $v^{-1}$. This is a consequence of the next proposition, which is analogous to Proposition 2.3 in \cite{KT16}.
\begin{proposition}\label{prop:min_vol}
	Fix $v \in \interior \stdCone$. Then $w = v^{-1}$ minimizes $\Vol(w,v)$. So that
	$$
	\Vol(w,v) = \left(\frac{r}{\sqrt[r]{\det v^{-1}}}\right)^{d }\Vol(\stdInt,\stdInt/r)= \left(r\sqrt[r]{\det v}\right)^{d}\Vol(\stdInt,\stdInt/r).
	$$
\end{proposition}	
\begin{proof}
	It follows from Proposition \ref{prop:vol} that minimizing $\Vol(w,v)$ is the same as 
	minimizing  $ \left(\frac{\inProd{w}{v}}{\sqrt[r]{\det w}}\right)^{d }$. We will minimize, equivalently, 
	$\log \left(\frac{\inProd{w}{v}}{\sqrt[r]{\det w}}\right)$. We have
	\begin{align*}
	f(w) = \log \left(\frac{\inProd{w}{v}}{\sqrt[r]{\det w}}\right) = \log\inProd{w}{v} - \frac{1}{r}\log \det w.
	\end{align*}
	Proposition \RNum{3}.4.2 of \cite{FK94} tells us that $\nabla \log \det w = w^{-1}$. We then have:
	\begin{align*}
	\nabla f(w) = \frac{v}{\inProd{w}{v}} -  \frac{1}{r} w^{-1}.
	\end{align*}
	In order to force $\nabla f(w) = 0$, we may take, for instance, $w = v^{-1}$ so 
	that $\inProd{w}{v} = r$. From Proposition \ref{prop:vol}, we obtain
	$$
	\Vol(w,v) = \left(\frac{r}{\sqrt[r]{\det v^{-1}}}\right)^{d }\Vol(\stdInt,\stdInt/r) = \left(r\sqrt[r]{\det v}\right)^{d}\Vol(\stdInt,\stdInt/r) .
	$$
\end{proof}
From Proposition \ref{prop:min_vol}, our quest for a good $H(w,v)$ can be summarized by trying 
to minimize 
$$
\det w^{-1} = \det(\alpha y + \beta \stdInt )^{-1},
$$ 
subject to $\alpha \geq 0, \beta \geq 0$ and $\alpha  \frac{\inProd{y}{\stdInt}}{\constInv}  + \beta \leq r$.
Note that there is no benefit in letting the last inequality be inactive. Furthermore, 
 we would rather minimize the logarithm of the function above, which gives a convex problem.
So we let $\alpha = \left(\frac{r-\beta}{\inProd{y}{\stdInt}} \right) \constInv$ and consider the following 
problem
\noindent\begin{align}
\underset{ 0 \leq \beta \leq r}{\inf} & \quad-\log \det\left(r\stdInt+  \left(\frac{r-\beta}{\inProd{y}{\stdInt}} \right) \constInv y + \stdInt\left(\beta - r\right) \right). \tag{$\mathrm{Aux}$} \label{eq:aux_prob}
\end{align}

Of course, there is no need to minimize \eqref{eq:aux_prob} exactly. We only need that $\frac{r}{\sqrt[r]{\det w}} < \delta$, or equivalently, 
$- r\log r + r\log \delta  > - \log \det w$, for some fixed positive constant $\delta < 1$ that does not depend on $\stdMap$ but could, possibly, 
depend on $r$ and $d$.
 In Theorem \ref{theo:step} we 
present a choice of $\beta$ that is good enough for our purposes. 

We first need the following auxiliary fact, which appears in some form or another in the analysis of many interior point methods dating back to the original paper by Karmakar, see Lemma 4.2 in \cite{K84}. It is in fact a consequence of the self-concordance of the $-\log \det(\cdot)$ function and can be derived from 
the bounds appearing, for instance, in Section 2.1.1 in \cite{NT08}. 
Still, for the sake of self-containment, we will not directly use self-concordance in proving the next bound.
\begin{lemma}\label{lemma:log}
	Let $h \in \jAlg$ be such that $\norm{h} < r$. Then
	\begin{align*}
		-\log \det (r \stdInt  + h) & \leq  - r \log r - \frac{\inProd{h}{\stdInt}}{r} + \frac{\norm{h}^2}{2r(r-\norm{h})}.
	\end{align*}
\end{lemma}
\begin{proof}
	Note that the case $r = 1$ corresponds to $\jAlg = \Re, \stdCone = \Re_{+}, \stdInt = 1$ and is the statement that if $\alpha \in \Re$ and $-1 < \alpha < 1$, then 
	\begin{equation}
	-\log (1  + \alpha)  \leq  - \alpha + \frac{{\alpha}^2}{2(1-\abs{\alpha})}, \label{eq:log}
	\end{equation}
	which is a known bound for the logarithm function. See Lemmas 4.1 and 4.2 in \cite{K84} or Lemma 3.1 in \cite{F02}. From \eqref{eq:log} we can 
	prove the general case. Let $\lambda _1, \ldots, \lambda _r$ be the eigenvalues of $h$. We have the following relations
	\begin{align}
		\inProd{h}{\stdInt} & = \lambda _1 + \cdots + \lambda _r \label{eq:log:aux1} \\
		\norm{h}^2 & = \lambda _1 ^2 + \cdots + \lambda _r^2. \label{eq:log:aux2}
	\end{align}
	Since the corresponding eigenvalues of $r\stdInt + h$ are $r + \lambda _1, \ldots, r + \lambda _r$, we have
	\begin{align}
		-\log \det (r \stdInt  + h) & = -r\log r - \sum _{i=1}^r   \log\left(1 + \frac{\lambda _i}{r}\right). \label{eq:log2}
	\end{align}
	By hypothesis, we have $\norm{h}^2 < r^2$, which implies that $\abs{\lambda _i / r}^2 < 1$  and $\abs{\lambda _i / r} < 1$  for every $i$.
	Recalling \eqref{eq:log:aux1} and applying \eqref{eq:log} to each term of \eqref{eq:log2} we get
	\begin{align*}
	-\log \det (r \stdInt  + h) & \leq -r\log r - \frac{\inProd{h}{\stdInt}}{r} + \sum _{i=1}^r \frac{\lambda _{i}^2}{2r^2(1- \abs{\lambda_i/r})}.
	\end{align*}
	Finally, due to \eqref{eq:log:aux2} and the fact that $1- \abs{\lambda_i/r} \geq 1- \norm{h}/r > 0$ holds for all $i$, we have the desired inequality. 
	\end{proof}
We are now ready to show how to construct $H(w,v)$ satisfying the properties \ref{p:1}, \ref{p:2} \ref{p:3} outlined in Section \ref{sec:int}.
\begin{theorem}\label{theo:step}
Let $\rho > 1$ and $y \in \stdCone$ be such that $\feasPs \subseteq H(y,\stdInt/\rho r)$ and $y \neq 0$. Let\footnote{If $y\neq 0$ and $y \in \stdCone$ then $\inProd{y}{\stdInt}> 0$, since the latter is sum of the eigenvalues of $y$, which are nonnegative. Furthermore, not all of them are zero, since $y \neq 0$.}
\begin{align*}
\beta & =  r- \left(\frac{1}{\rho} - \frac{1}{\sqrt{\rho(3\rho - 2)}}\right)\\
w & = \left(\frac{r-\beta}{\inProd{y}{\stdInt}} \right) \constInv y + \beta \stdInt\\
v & = w^{-1}.
\end{align*}
Then, the following hold:
\begin{enumerate}[label=({\it \roman*})]
	\item $\feasPs \subseteq H(y,\stdInt/\rho r)\cap H(\stdInt,\stdInt/r) \subseteq H(w,v)$
	\item $Q_{w^{-1/2}\sqrt{r}}(H(\stdInt,\stdInt/r))  = H(w,v)$
	\item $\Vol(w,v) = \left(\frac{r}{\sqrt[r]{\det w}}\right)^{d }\Vol(\stdInt,\stdInt/r) \leq \left( {\exp({-\varphi(\rho)/r})}\right)^{d} \Vol(\stdInt,\stdInt/r)$, where 
	$$
	\varphi(\rho) = 2 - \frac{1}{\rho} - \sqrt{3- \frac{2}{\rho}}  .
	$$
	In particular if $\rho \geq 2$, we have $\Vol(w,v) <  \left( 0.918 \right)^{d/r}\Vol(\stdInt,\stdInt/r)$.
\end{enumerate}
\end{theorem}
\begin{proof}
Following the discussion so far, we consider the problem \eqref{eq:aux_prob}. 
\noindent\begin{align}
\underset{ 0 \leq \beta \leq r}{\inf} & \quad-\log \det\left(r\stdInt+  \left(\frac{r-\beta}{\inProd{y}{\stdInt}} \right) \constInv y + \stdInt\left(\beta - r\right) \right). \tag{$\mathrm{Aux}$} 
\end{align} 
To avoid making it seem as if $\beta$ appeared from thin air, we will try to show the rationale behind 
our choice of $\beta$.
Our task is to show that there is some $0 \leq \beta \leq r$ and some fixed constant $\delta < 1$ for which
\begin{equation*}
\frac{r}{\sqrt[r]{\det w}} < \delta.
\end{equation*}
Or, equivalently,
\begin{align*}
-\log \det\left(r\stdInt+  \left(\frac{r-\beta}{\inProd{y}{\stdInt}} \right) \constInv y + \stdInt\left(\beta - r\right) \right) \leq -r\log r + r\log \delta.
\end{align*} 
Once we find  $\beta$ and $\delta$ we can claim that $\Vol(w,v) \leq \delta ^{d}\Vol(\stdInt,\stdInt/r)$. 
 Furthermore, by construction, 
we have item $(i)$, since $\beta$ and $\alpha = \left(\frac{r-\beta}{\inProd{y}{\stdInt}} \right) \constInv$ are nonnegative and solve the linear 
inequalities \eqref{eq:hyper1} and \eqref{eq:hyper2}. Item $(ii)$ is just a consequence 
of Proposition \ref{prop:vol}.
To start, let $h = \left(\frac{r-\beta}{\inProd{y}{\stdInt}} \right) \constInv y + \stdInt\left(\beta - r\right)$. We have
\begin{align}
h & =   \left(r - \beta  \right)\left( \frac{\constInv}{\inProd{y}{\stdInt}} y - \stdInt\right) \notag \\
\frac{\inProd{h}{\stdInt}}{r} & =   (r - \beta)(\rho - 1) \label{eq:in_bd}\\
\norm{h}^2 & = \left( r - \beta  \right)^2 \left(r(1-2\rho) +  \rho^2r^2 \frac{\norm{y}^2}{\inProd{y}{\stdInt}^2}\right) \notag
\end{align}
As $ \norm{y} \leq \inProd{y}{\stdInt}$ and $\rho > 1$, we get the bound 
\begin{equation}
\norm{h} \leq (r -\beta )\rho r. \label{eq:h_bd}
\end{equation}
 So in order to apply 
Lemma \ref{lemma:log}, it is enough to take $(r- \beta)\rho r < r$. So, suppose that  $r-\beta < 1/\rho$. This, together with 
\eqref{eq:h_bd}, implies that
\begin{equation}
 \frac{1}{r(1 - \rho(r -\beta ))} \geq \frac{1}{r - \norm{h}}. \label{eq:h_bd2}
\end{equation}
Then, \eqref{eq:in_bd}, \eqref{eq:h_bd} and \eqref{eq:h_bd2} together with Lemma \ref{lemma:log} imply that:
\begin{align*}
-\log \det (r \stdInt  + h)  & \leq  - r \log r - (r - \beta)(\rho - 1) + \frac{(r-\beta)^2\rho^2r^2}{2r^2(1 - \rho(r -\beta ))} \\
& = - r \log r - (r - \beta)(\rho - 1) + \frac{(r-\beta)^2\rho^2}{2(1 - \rho(r -\beta ))}.
\end{align*}
Let $\psi(z) = -z(\rho - 1) + \frac{z^2\rho ^2}{2(1-\rho z)}$. A tedious but straightforward 
calculation\footnote{It is easy to verify this computation by using a computer algebra 
	system such as the open-source software \texttt{maxima} \cite{maxima}.  Consider the following three \texttt{maxima} commands: {\texttt{f:-z*(p-1)+(z\textasciicircum 2*p\textasciicircum 2)/(2*(1-p*z));
		zp:solve(diff(f,z,1),z);
		for i:1 thru 2 do disp(combine(expand(ratsimp(substitute(zp[i],z,f)))));}}. The first two commands compute $z_{\rho}$ and the last computes $\psi (z_{\rho})$. Note that we discard one of the solutions because it is negative.}  shows that in the interval $(0,1/\rho)$, $\psi(z)$ is minimized at 
$$z_\rho = \frac{1}{\rho} - \frac{1}{\sqrt{\rho(3\rho - 2)}}. $$ 
and for 
that $z_\rho $ we have $ \psi(z_\rho) = \sqrt{3- \frac{2}{\rho}} + \frac{1}{\rho} - 2$. So, we let $r- \beta = z_\rho$ and conclude that for this choice of $\beta$, we 
have  
\begin{align*}
-\log \det (r \stdInt  + h)  & \leq - r \log r - \left(2 - \frac{1}{\rho} - \sqrt{3- \frac{2}{\rho}}   \right).
\end{align*}
Therefore, we take $\delta$ such that $r\log \delta = \psi(z_\rho)$.
Thus we conclude that $\Vol(w,v) \leq \left( {\exp({-\varphi(\rho)/r})}\right)^{d}\Vol(\stdInt,\stdInt/r)$, 
where $\varphi(\rho) = -\psi(z_\rho)$. Examining the first derivative of 
$-\varphi$, we see that it is a decreasing function. In particular, 
if $\rho \geq 2$, we have $\exp({-\varphi(\rho)}) < 0.918$.
\end{proof}
\begin{figure}
	\centering
\includegraphics[scale=0.5]{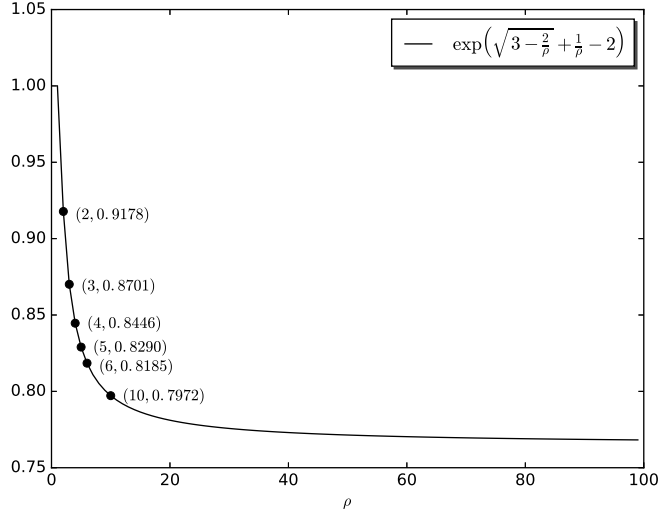}
\caption{The function $\exp(-\varphi(\rho))$. The decrease in 
	volume is bounded by $\exp(-\varphi(\rho))^{\frac{d}{r}}$.}\label{fig:vol}
\end{figure}

See Figure \ref{fig:vol} for the graph of $\exp({-\varphi(\rho)})$.
\subsection{The general case}\label{sec:vol:general}
Now, we move on to the general case and we suppose that $\stdCone = \stdCone _1 \times \cdots \times \stdCone _\ell$, where each $\stdCone _i$ is 
a simple symmetric cone of rank $r_i$ and dimension $d_i$. We have $r = r_1 + \ldots + r_\ell$ and $d = d_1 + \ldots + d_\ell$. Suppose
that we have found some nonzero $y \in \stdCone$ satisfying 
$$
\norm{P_\stdMap  y} \leq \frac{1}{2r_{\max}\sqrt{\ell}}\norm{y}_{{1,\infty}},
$$
where we recall that $r_{\max} = \max \{r_1,\ldots, r_\ell \}$ and that since $y \in \stdCone$, we have $\norm{y}_{{1,\infty}} = \max\{\inProd{\stdInt_1}{y_1},\ldots,\inProd{\stdInt_\ell}{y_\ell}\}$, see Section \ref{sec:norms}. We have the following observation.

\begin{lemma}\label{lemma:y_proj2}
	Suppose $x$ is feasible for \eqref{eq:feas_p_s} and that $y \in \stdCone$ satisfies $P_\stdMap y \neq 0$ and
		$$
		\norm{P_\stdMap  y} \leq \frac{1}{2r_{\max} \sqrt{\ell}} \norm{y}_{{1,\infty}}.
		$$
	For every $k \in \{1,\ldots, \ell\}$, define $\rho _k$ as 
	\begin{equation}
	\rho _k = \frac{\norm{y_k}_1}{ r_k\norm{P_\stdMap  y}\sqrt{\ell}}. \label{eq:rho_k}
	\end{equation}
\begin{enumerate}[label=({\it \roman*})]
	\item Let $k$ be any index such that $\norm{y}_{1,\infty} = \norm{y_k}_{1}$, then $\rho _k \geq 2$.
	\item Let $x$ be feasible for \eqref{eq:feas_p_s}, then  $x_k \in H(y_k,\stdInt_k/\rho _k r_k)$.
\end{enumerate}
	
\end{lemma}
\begin{proof}
\begin{enumerate}[label=({\it \roman*})]
	\item By the definition of $\rho_k$, we see that $\rho _k \geq 2$ if $\norm{y_k}_1 = \norm{y}_{1,\infty}$, since $r_k \leq r_{\max}$.
	
	\item Since $y \in \stdCone$, we have $\norm{y_k}_{1} = \inProd{y_k}{\stdInt_k}$. 
Recall that $P_\stdMap$ is self-adjoint, so that $ \inProd{y}{ P_\stdMap  x} =  \inProd{P_\stdMap y}{  x} $. In addition, since $x \in  \feasPs$, we have $\norm{x}_{1,\infty} \leq 1$. Therefore, using 
	the  Proposition \ref{prop:inequality} together with the generalized Cauchy-Schwarz inequality in \eqref{eq:1_inf} we have:
	\begin{align}
	\inProd{y_k}{x_k} \leq \inProd{y}{ x} = \inProd{y}{ P_\stdMap  x} \leq \norm{P_\stdMap  y}_{\infty,1} \norm{x}_{1,\infty} \leq   \norm{P_\stdMap  y}_{2}\sqrt{\ell} =  \frac{\norm{y_k}_1}{\rho_k r_k} = \frac{\inProd{y_k}{\stdInt_k}}{\rho_k r_k}. \notag
	\end{align}
	\end{enumerate}

	\end{proof}
As a reminder, we write $\Vol(w_k,v_k)$ for the volume of the region 
$H(w_k,v_k) \cap \stdCone _k$.
The idea is that once $y$ is found, we can use $y_k$  and Theorem \ref{theo:step} to find a hyperplane $H(w_k,v_k) \subseteq \jAlg _k$, such that  $\Vol(w_k,v_k)$ is small. 
This is summarized in the following theorem.

\begin{theorem}\label{theo:vol2}
Let $y \in \stdCone$ be such that 	$
\norm{P_\stdMap  y} \leq \frac{1}{2r_{\max} \sqrt{\ell}} \norm{y}_{{1,\infty}}$ and that $P_\stdMap y \neq 0$. 
For all blocks $\jAlg _k$, let $\rho _k$ be as in Equation \eqref{eq:rho_k} and suppose that 
for some $k$ we have $\rho _k > 1$. Let $w_k,v_k \in \stdCone _k, \beta_k \in \Re$ be such that\footnote{$P_\stdMap y \neq 0$ implies $y \neq 0$, which implies $\inProd{y}{\stdInt}> 0$, since $y \in \stdCone$. See the footnote in Theorem \ref{theo:step}.}
\begin{align*}
\beta_k & =  r_k- \left(\frac{1}{\rho _k} - \frac{1}{\sqrt{\rho _k(3\rho _k - 2)}}\right)\\
w_k & = \left(\frac{r_k-\beta_k}{\inProd{y_k}{\stdInt_k}} \right) \rho _k r_k y_k + \beta_k \stdInt _k\\
v_k & = w_k^{-1}.
\end{align*}
Then, the following hold:
\begin{enumerate}[label=({\it \roman*})]
	\item If $x \in \feasPs$, then $x_k \in  H(w_k,v_k)$,
	\item $Q_{w_k^{-1/2}\sqrt{r_k}}(H(\stdInt_k,\stdInt_k/r_k))  = H(w_k,v_k)$,
	\item $\Vol(w_k,v_k) = \left(\frac{r_k}{\sqrt[r_k]{\det w_k}}\right)^{d_k }\Vol(\stdInt_k,\stdInt_k/r_k) \leq \left( {\exp({-\varphi(\rho_k)/r_k})}\right)^{d_k} \Vol(\stdInt_k,\stdInt_k/r_k)$, where 
	$$
	\varphi(\rho _k) = 2 - \frac{1}{\rho}  - \sqrt{3- \frac{2}{\rho}}.
	$$
	In particular if $\rho _k \geq 2$, we have $\Vol(w_k,v_k) <  \left( 0.918 \right)^{d_k/r_k}\Vol(\stdInt_k,\stdInt_k/r_k)$.
\end{enumerate}
\end{theorem}
\begin{proof}
Due to Lemma  \ref{lemma:y_proj2}, we have that if $x \in \feasPs$, then 
$x_k$ is confined to $H(y_k,\stdInt_k/\rho_k r_k) \cap H(\stdInt_k, \stdInt_k/r_k)$.
Then, the choice of $\beta_k, w_k, v_k$ is exactly the same as suggested by Theorem \ref{theo:step}.
This shows that $H(y_k,\stdInt_k/\rho_k r_k) \cap H(\stdInt_k, \stdInt_k/r_k) \subseteq H(w_k,v_k)$, which implies $x_k \in H(w_k,v_k)$, which is item $(i)$. Item $(ii)$ and $(iii)$ are also 
direct consequences of items $(ii)$ and $(iii)$ of Theorem \ref{theo:step}.
\end{proof}
Then, the idea is to replace $\stdMap$ by $\stdMap Q$, where $Q$ is the linear map such that 
$$
Q(x_1,\ldots,x_k,\ldots,x_\ell) = \left(Q_1(x_1),\ldots,Q_\ell (x_\ell)\right),
$$
where 
$$
Q_{i} = \begin{cases}
Q_{w_i^{-1/2}\sqrt{r_i}} & \text{if} \quad  \rho _i > 1 \\
I _i & \quad \text{otherwise}, \\
\end{cases}
$$
where $I_i$ is the identity map on $\jAlg _i$. If $y$ satisfying Proposition 
\ref{theo:vol2} is found, $\rho _k \geq 2$ holds for at least one block. Therefore, 
for at least one block the decrease in volume is bounded above by a positive constant 
less than $1$. 
This is the foundation upon which Algorithm \ref{alg:main} in Section \ref{sec:full} is constructed.

\section{Basic Procedure}\label{sec:bp}
We remark that the procedure described here is a direct extension of Chubanov's basic procedure in 
\cite{Ch15}, see also  Section 3 in \cite{KT16} for 
the corresponding discussion on SOCPs.
From Section \ref{sec:vol} and Theorem \ref{theo:vol2}, we see that as long as we are able to find $y$ such that $\norm{P_\stdMap  y} \leq \frac{1}{2r_{\max}\sqrt{\ell}}\norm{y}_{{1,\infty}}$, we can confine one of the blocks of the feasible solutions of \eqref{eq:feas_p_s} to a region 
$H(w_k,v_k)\cap \stdCone_k$ with smaller volume than $H(\stdInt_k,\stdInt_k/r_k)\cap \stdCone_k$.

The ``basic procedure'' can be seen as an algorithm for minimizing $\norm{P_\stdMap  y}$, but we early stop it and pass control forward after the threshold $\norm{P_\stdMap  y} \leq \frac{1}{2r_{\max}\sqrt{\ell}}\norm{y}_{{1,\infty}}$ is met. 
So suppose that we start with some $y \in \interior \stdCone$ such that $\inProd{y}{\stdInt} = 1$.
We then let $z = P_\stdMap  y$.

If $z  = 0$, then $y \in (\ker \stdMap)^\perp$, so $y = \stdMap^\T u$ for some $u$, which means that $y$ is feasible for \eqref{eq:feas_d} since $y \in \stdCone$ and $y$ is not zero. If $z \in \interior \stdCone$, then since $\stdMap z = 0$, we have that $z$ is feasible for \eqref{eq:feas_p}. If 
neither of those criteria are met but $\norm{z} \leq \frac{1}{2r_{\max}\sqrt{\ell}}\norm{y}_{1,\infty}$ holds, then we stop the algorithm anyway.

Otherwise, we have $z  \neq 0$ and $z \not \in \interior \stdCone$. Because $z \not \in \interior \stdCone$, 
there is a nonzero $c \in \stdCone$ such that $\inProd{z}{c} \leq 0$, which is a consequence of standard separation results and of the 
fact that $\stdCone$ is self-dual. Since any multiple of $c$ will do the job, we can scale it and suppose that 
$\inProd{\stdInt}{c} = 1$. 
Then, it seems sensible to try to correct $y$ by considering 
$$
y' = \alpha y + (1-\alpha) c,
$$
where we select $\alpha$ in such a way that $P_\stdMap y'$ is as close as possible to the origin, since we wish to minimize $\norm{P_\stdMap y}$. Let $p = P_\stdMap c$. The $\alpha$ that minimizes $\norm{P_\stdMap y'}$ is 
$$
\alpha=\frac{\inProd{p}{p - z}}{\norm{z-p}^2}.
$$
Note that after computing $p$, we may check whether $p$ meets our stopping criteria. That is, 
we check whether $p = 0$ (so that $c$ is feasible for \eqref{eq:feas_d}), $p \in \interior \stdCone$ (so that $p$ is feasible for \eqref{eq:feas_p}) or $\norm{p} \leq \frac{1}{2r_{\max}\sqrt{\ell}}\norm{c}_{1,\infty}$. If neither $p$ nor $z$ satisfies our stopping criteria, then $\alpha \in (0,1)$ so that $y' \in \interior \stdCone$, see 
Theorem 6.1 in \cite{rockafellar}. Moreover, this update decreases the norm of $z$, as described by the following proposition.
\begin{proposition}\label{prop:bp}
Let:
\begin{enumerate}[label=({\it \roman*})]
	\item $y \in \interior \stdCone $, $\inProd{\stdInt}{y} = 1$ and $z = P_\stdMap y$,
	\item $c \in \stdCone$ be such that $\inProd{z}{c} \leq 0$ and $\inProd{\stdInt}{c} = 1$,
	\item $p = P_\stdMap c$ and $\alpha=\frac{\inProd{p}{p - z}}{\norm{z-p}^2}$,
	\item $y' = \alpha y + (1-\alpha)c$.
\end{enumerate}
Suppose that $p\neq 0$ and $z \neq 0$. Then:
\begin{enumerate}[label=({\alph*})]
	\item $\frac{1}{\norm{P_\stdMap y'}^2} \geq \frac{1}{\norm{P_\stdMap y}^2} + 1$.
		\item $y' \in \interior \stdCone $, $\inProd{\stdInt}{y'} = 1$ and $ \norm{y'}_{1,\infty} \geq \frac{1}{\ell}$,
\end{enumerate}
\end{proposition}
\begin{proof}
	We first prove item $(a)$. 	By $(ii)$, we have
	$$
	\inProd{z}{p} = \inProd{P_\stdMap z}{c} = \inProd{z}{c} \leq 0. 
	$$
	This implies that  
	\[
	\norm{z-p}^2 = \norm{z}^2 + \norm{p}^2 - 2\inProd{z}{p} > 0,
	\]
	since both $z$ and $p$ are nonzero.
	We also have
	\[
	1 = \frac{\norm{z}^2 -\inProd{z}{p}}{\norm{z-p}^2}  + \frac{\norm{p}^2 - \inProd{z}{p}}{\norm{z-p}^2}. 
	\]
	We then conclude that the denominator of $\alpha$ is not zero and $\alpha \in (0,1)$.	 
	We then have 
$$
	{P_\stdMap { y'}} =p + \alpha (z - p).
$$
	Therefore, 
	\[
	 \norm{P_\stdMap {y'}}^2 = \alpha^2\norm{z - p}^2+2\alpha \inProd{p}{z- p} + \norm{p}^2.
	\]
	It follows that
	\[
	\norm{P_\stdMap {y'}}^2 = \norm{p}^2 - \frac{\inProd{p}{z - p}^2 -\inProd{z}{p}^2}{\|z-p\|^2}  = \frac{\|z\|^2\|p\|^2 - \inProd{z}{p}^2}{\|z\|^2+\|p\|^2-2\inProd{z}{p}} \leq \frac{\|z\|^2\|p\|^2}{\|z\|^2+\|p\|^2}.
	\]
	Since $P_\stdMap$ is a projection matrix, we have $\|p\|^2\leq\|c\|^2 \leq \inProd{c}{\stdInt}^2 \leq 1$.
	This implies that
	\[
	\frac1{\norm{P_\stdMap {y'}}^2}\geq \frac1{\|z\|^2} + \frac1{\|p\|^2} \geq \frac1{\|z\|^2} + 1.
	\]

Now we prove item $(b)$. Since we have shown that $\alpha \in (0,1)$, $y'$ is a strict convex combination of 
$c$ and the interior point $y$, so it is also an interior point. For the same reason, $y'$ also satisfies 
$\inProd{\stdInt}{y'} = 1$, since $\inProd{\stdInt}{c} = 1$ and $\inProd{\stdInt}{y} = 1$, due to items $(i)$ and $(ii)$. Finally, note that since $y' \in \stdCone$, we have 
$1 = \inProd{y'}{\stdInt} = \sum _{i=1}^\ell \norm{y'_i}_1 \leq \ell  \norm{y'}_{1,\infty} $, which 
shows that $\norm{y'}_{1,\infty} \geq \frac{1}{\ell}$.
\end{proof}
We remark that although the increase in $\frac1{\norm{P_\stdMap {y'}}^2}$ is  bounded below by $1$ regardless of 
$c$, the proof shows that the actual increase is a function of $-\inProd{z}{p} = -\inProd{z}{c}$. Therefore, if we wish 
to maximize the increase, we have to minimize $\inProd{z}{c}$ subject to $\inProd{c}{\stdInt} = 1$ and $c \in \stdCone$.
Similarly to the case of positive semidefinite matrices, the optimal value of this problem is $\lambda _{\min}(z)$, which is 
achieved when $c$ is the idempotent associated to the minimum eigenvalue of $z$, see, for instance, Equation (9) in \cite{Sturm2000}.

As we are assuming that $\stdCone = \stdCone _1 \times \ldots \times \stdCone _\ell$, the computation of the minimum eigenvalue of $z$ can be done by computing the minimum eigenvalue of each block and then taking the overall minimum. That is, 
$$
\lambda _{\min}(x) = \min \{  \lambda_{\min}(x_1), \ldots, \lambda _{\min} {(x_\ell)} \}.
$$
This is advantageous because if $\stdCone _i$ is either $\Re_+$ or $\SOC{\ell}$ the minimum eigenvalue and the corresponding idempotent can be computed exactly, see \eqref{eq:eig_soc} and \eqref{eq:eig_soc2}.

We can now state our version of the basic procedure  which encompasses the discussion so far, see Algorithm \ref{alg:basic}. We remind 
that superscripts in Algorithms \ref{alg:basic}, \ref{alg:main}, \ref{alg:basic:sdp} and \ref{alg:main:sdp} such as $y^i, \stdMap^i$ denote the iteration number, not exponentiation.We have the following complexity bounds. For the SOCP analogue, see the comments before Section 5 in \cite{KT16}.
{
\begin{algorithm}
	\caption{Basic Procedure }\label{alg:basic}
	\KwIn{$\stdMap$, $P_{\stdMap}$, $y^1 \in \interior \stdCone$ such that $\inProd{y^1}{\stdInt} = 1$}
	\KwOut{
		\begin{enumerate}
			\item $y$ such that $\norm{P_\stdMap  y} \leq \frac{1}{2r_{\max}\sqrt{\ell}}\norm{y}_{{1,\infty}}$, or
			\item a solution to \eqref{eq:feas_p}, or
			\item a solution to \eqref{eq:feas_d}.
		\end{enumerate}}
		
$ i \leftarrow 1$, $z \leftarrow P_\stdMap  y^1$. \label{alg:basic:proj}\\
\While{$z \neq 0$ and $z \not \in \interior \stdCone$ and $\norm{z} > \frac{1}{2r_{\max}\sqrt{\ell}}\norm{y^i}_{{1,\infty}}$\label{alg:bp:stop}}{
Let $c$ be an idempotent such that $\inProd{\stdInt}{c} = 1$ and $\inProd{z}{c} = \lambda _{\min}(z)$.\label{alg:basic:ide}\\
$p \leftarrow P_\stdMap c$.\\
\eIf{$p = 0$ or $p \in \interior \stdCone$ or $\norm{p} \leq \frac{1}{2r_{\max}\sqrt{\ell}}\norm{c}_{{1,\infty}} =  \frac{1}{2r_{\max}\sqrt{\ell}}$}
{\textbf{stop} and \textbf{return}  $c$.}
{$y^{i+1} \leftarrow \alpha y^i + (1-\alpha) c$, where $\alpha=\frac{\inProd{p}{p - z}}{\norm{z-p}^2}$. \\}
Set $i \leftarrow i+1$ and $z \leftarrow P_\stdMap  y^i$.
}
\textbf{return} $y^i$.
\end{algorithm}
}

\begin{proposition}\label{prop:bp_c}
Algorithm \ref{alg:basic} needs no more than $4\ell^3 r_{\max}^2$ steps before it halts.	
\end{proposition}
\begin{proof}
Due to Proposition \ref{prop:bp}, all the iterates $y^i$ satisfy $\norm{y^i}_{1,\infty} \geq 1/\ell$. 
Therefore, if $\norm{P_\stdMap y^i} \leq \frac{1}{2r_{\max}\ell^{3/2}}$, we will meet the stopping criteria 
$\norm{P_\stdMap y^i} \leq \frac{\norm{y^i}_{1,\infty}}{2r_{\max}\sqrt{\ell}}$. Also due to Proposition \ref{prop:bp}, $\frac{1}{\norm{P_\stdMap  y^i}^2}$ improves by at least one 
after each iteration. Therefore, after $4\ell^3 r_{\max}^2$  iterations, Algorithm \ref{alg:basic} will halt for sure. 
\end{proof}

Any algorithm aimed at minimizing $\norm{P_\stdMap  y} $ can serve as a ``basic procedure'', as long as there is some bound 
that ensures that we will hit one of the stopping criteria in a finite amount of time. In Section 5 of \cite{PS16}, the authors describe 
four different algorithms that can serve as basic procedure. 
The algorithm described here is analogous to the ``Von Neumann scheme'' which, as the authors remark, is essentially what 
was used by Chubanov in \cite{Ch15}. Nevertheless among the four, the  ``smooth perceptron'' \cite{SP12} has the better complexity bound for the number of 
iterations, although each iteration is more expensive.  

In particular, the smooth perceptron is such that if it has not finished 
at the $i$-th iteration then, $\norm{Py^i}^2 \leq \frac{8}{(t+1)^2}$. So if we use  the stopping criteria that $y$ should satisfy 
$\norm{P_\stdMap  y} \leq \frac{1}{2r_{\max}\sqrt{\ell}}\norm{y}_{{1,\infty}}$, as in the proof of Proposition \ref{alg:basic}, it is 
enough to force $\norm{P_\stdMap  y}$ to be smaller than $\frac{1}{2r_{\max}\ell^{3/2}} $, which will happen in no more than 
$4\sqrt{2}r_{\max} \ell ^{3/2} -1 $ iterations.

We will now estimate the cost per iteration of Algorithm \ref{alg:basic}. 
Recall that the dimension of $\jAlg$ is $d$, so we can identify $\jAlg$ with 
some $\Re^d$.  Then, we can assume that $\stdMap$ is, in fact, a $m\times d$ matrix.
For the sake of estimating the computational cost, we will assume that $\stdMap$ is 
surjective, but we emphasize that the theoretical results proved in this work do not need this assumption. 
Note that  computations performed in both the Basic Procedure (Algorithm \ref{alg:basic}) and the Main Procedure (Algorithm \ref{alg:main}) only depend on the kernel of $\stdMap$.

Assuming that $\stdMap$ is surjective, we may compute $P_{\stdMap}$ naively through 
the expression $I - \stdMap^*(\stdMap\stdMap^*)^{-1}\stdMap$, where $I$ is the identity operator on $\jAlg$. 
Following a similar suggestion in Section 6 of \cite{PS16}, 
we compute the Cholesky decomposition of $(\stdMap\stdMap^*)^{-1}$ so that $(\stdMap\stdMap^*)^{-1} = LL^*$, where $L$ is a $m \times m$ matrix. Then, we store the $m\times d$ matrix $L^*\stdMap$. 
Note that computing 
$\stdMap\stdMap^*$ and decomposing its inverse can be done in 
time proportional to $\mathcal{O}(m^3) + \mathcal{O}(m^2d)$.

 As $P_{\stdMap}$ stays constant throughout the algorithm, we only need to 
compute it once and then we just have to worry about computing $P_{\stdMap}y^i$ and $ P_{\stdMap}c $ at each iteration,  which costs $\mathcal{O}(md)$ if we make efficient use of the matrix $L^*\stdMap$. 
In our case, $P_{\stdMap}$ is computed inside Algorithm 
\ref{alg:main} and Algorithm \ref{alg:basic} receives it as input.

We now list other costs that need to be taken into account: the cost $c_{\interior \stdCone}$ of deciding 
if $z \in \interior \stdCone$, the cost $c_{\text{min}}$ of computing  the minimum eigenvalue $\lambda _{\min}(x)$ with the corresponding idempotent and the cost $c_{\text{norm}}$ of computing the norm
of $\norm{x}_{1,\infty}$ for elements $x \in \stdCone$. 

The cost $\norm{x}_{1,\infty}$ is proportional to $\mathcal{O}(d)$ since for $x \in \stdCone$, 
we have $\norm{x}_{1,\infty} = \max\, \{\inProd{x_1}{\stdInt_1},\ldots,\inProd{x_{\ell}}{\stdInt _{\ell} } \}$.
The cost $c_{\interior \stdCone}$ is majorized by $c_{\text{min}}$, since we can decide 
whether $x$ lies in interior of $\stdCone$ by checking if $\lambda _{\min}(x) > 0$. 
So the complexity of the basic procedure is proportional to
\begin{equation}\label{eq:bp:cost}
\mathcal{O}(\ell^3 r_{\max}^2(\max(md,c_{\text{min}}))).
\end{equation}

To conclude this section, we note that the condition ``$z \neq 0$'' in Line \ref{alg:bp:stop} of Algorithm \ref{alg:basic} can, optionally, be substituted by ``$y^i -z\not \in \stdCone$''. 
The motivation is that if $y^i -z \in \stdCone$, then there are two possibilities.
The first is $y^i - z \neq 0$ and $y^i - z$ is a solution to \eqref{eq:feas_d}, since $y^i - z $ belongs to the range of $\stdMap^*$.
The second is $y^i - z = 0$ and $z$ is a solution to \eqref{eq:feas_p}, since $y^i$ belongs to $\interior \stdCone$ throughout the algorithm.
Either case, we can stop the algorithm.
 
Using the condition ``$y^i -z\not \in \stdCone$'' does not affect the analysis conducted in this section because $y^i -z\not \in \stdCone$ implies $z \neq 0$, since $y^i \in \interior \stdCone$. Therefore, 
Propositions \ref{prop:bp} and \ref{prop:bp_c}  still apply.
The advantage is that since it is harder to satisfy than $z \neq 0$, it may lead to less iterations.
The  disadvantage is that, in order to verify that $y^i -z\not \in \stdCone$, it is necessary to do an extra minimum eigenvalue computation per iteration. 
This, however, does not alter the order complexity.

\section{Full-Procedure}\label{sec:full}
In this section, we summarize everything we have done so far and present an algorithm 
for the feasibility problem \eqref{eq:feas_p}, see Algorithm \ref{alg:main}. In essence, we call the basic procedure (Algorithm \ref{alg:basic}) and if we fail to find a solution to either \eqref{eq:feas_p} or \eqref{eq:feas_d},
we will receive some vector $y$. 
Following the discussion in Section \ref{sec:vol} and in Theorem \ref{theo:vol2}, for at least one index $k$, we are able to construct a half-space $H(w_k,v_k)$ such that the volume $H(w_k,v_k) \cap \stdCone _k$ is less than $H(\stdInt_k,\stdInt_k/r_k)\cap \stdCone_k$ and their ratio is bounded above by a constant, namely, 
${\exp({-\varphi(\rho_k)})}^{d_k/r_k} \leq \left( 0.918\right)^{d_k/r_k}$. Then, we 
construct an automorphism of $\stdCone$ that maps the region $H(\stdInt _k,\stdInt _k/r_k) \cap \stdCone _k$ to $H(w_k, v_k)\cap \stdCone_k$. Finally, we substitute $\stdMap$ by 
$\stdMap Q$ and repeat.

In  Algorithm \ref{alg:main}, we keep track of the volume reduction along the blocks through 
the $\epsilon _k$ variables. 
Lemma \ref{lemma:stop} tells us that the $\epsilon _k$ are upper bounds for 
the minimum eigenvalue of the feasible solutions to \eqref{eq:feas_p_s}.
When the accumulated reduction is sufficiently small, we have a certificate 
that \eqref{eq:feas_p_s} does not have $\epsilon$-feasible solutions. 

\begin{algorithm}[H]
	\caption{Main Algorithm}\label{alg:main}
	\KwIn{$\stdMap$, $\stdCone$, $\epsilon$}
	\KwOut{\begin{enumerate}
			\item a solution to \eqref{eq:feas_p}, or
			\item a solution to \eqref{eq:feas_d}, or
			\item ``there is no $\epsilon$-feasible solution''.
		\end{enumerate}}
$\stdMap ^1 \leftarrow \stdMap$, $i \leftarrow 0$ and 
$\epsilon_j \leftarrow 0$ for all $j \in \{1,\ldots, \ell\}$. \label{alg:main:first}\\
Compute $P_{\stdMap}$ and call the basic procedure (Algorithm \ref{alg:basic}) with $\stdMap^i,P_{\stdMap}, \frac{\stdInt}{r}$ and 
denote its output by $y$. \label{alg:main:call_bp}\\
\If{$y$ is feasible for \eqref{eq:feas_d} or \eqref{eq:feas_p}}{\textbf{stop} and \textbf{return} $y$.\label{alg:main:stop1}}
\For{$k\in \{1,\ldots, \ell \}$}{
	$\rho _k \leftarrow \frac{\norm{y_k}_1}{r_k \norm{P_\stdMap  y}\sqrt{\ell}}$ \label{alg:main:rho}\\
	 \eIf{$\rho _k > 1$ \label{alg:main:if}}
	{
		$\beta \leftarrow r_k- \left(\frac{1}{\rho_k} - \frac{1}{\sqrt{\rho_k(3\rho_k - 2)}}\right)$\\
		$w_k \leftarrow  \left(\frac{r_k-\beta}{\inProd{y_k}{\stdInt_k}} \right) \rho _k r_k y_k + \beta \stdInt _k$\label{alg:main:wk}\\
		$Q_k \leftarrow Q_{w^{-1/2}_k\sqrt{r_k}}$ \label{alg:main:qwk} \\
		$\epsilon _k \leftarrow \epsilon_k + \log r_k -  \frac{1}{r_k}\log {\det w_k} $\\
		\If{$\epsilon _k < \log r_k + \log \epsilon$ \label{alg:main:stop}}{\textbf{stop}, there is no $\epsilon$-feasible solution. }
		}
	{
		$Q_k \leftarrow I_k$
		}	
}
Let $Q^i$ be such that for all $x \in \jAlg$ we have $
Q^i(x)  = \left(Q_1(x_1),\ldots,Q_\ell (x_\ell)\right). 
$\\
$\stdMap^{i+1} \leftarrow \stdMap^i Q^i,$ $i\leftarrow i+1$. Go to Line \ref{alg:main:call_bp}. \label{alg:main:update_a}
\end{algorithm}
Algorithm \ref{theo:main} and the one in \cite{KT16} have a few differences.
First of all, for the same $y$, we attempt to find volume reductions along all blocks, where 
in \cite{KT16}, volume reductions only happen for exactly one block per iteration.
Furthermore, here we adapt a greedy approach and following Theorems \ref{theo:step} and \ref{theo:vol2}, 
we try to reduce the volume as much as possible by selecting $\rho _k$ to be as large as our analysis permits.

We will now prove a few lemmas regarding the correctness of Algorithm \ref{alg:main}.

\begin{lemma}\label{lemma:feas1}
For every iteration $i$, every block $\jAlg _k$ and for all $x \in \feasPs$ we have 
$$
x_k \in Q^1_k\times \cdots \times Q^i_k H(\stdInt_k,\stdInt_k/r_k).
$$
\end{lemma}
\begin{proof}
We will proceed by induction, so first we prove the result for $i = 1$. According to Line \ref{alg:main:if}  of Algorithm 
\ref{alg:main}, we have that $Q^1_k$ is either the identity map or 
$Q^1 _k = Q_{w_k^{-1/2}\sqrt{r_k}}$, where $w_k$ is selected as in 
Line \ref{alg:main:wk}. If it is the former, it is clear that 
$x_k \in Q^1_k H(\stdInt_k,\stdInt_k/r_k)$, since $\inProd{x_k}{\stdInt _k} \leq \norm{x}_{1,\infty} \leq 1$. If it is the latter, 
we must have $\rho_k > 1$, in which case,  Theorem \ref{theo:vol2} tells
us that $x_k \in H(w_k,w_k^{-1})$ and that $Q^1 _k H(\stdInt_k,\stdInt_k/r_k) = H(w_k,w_k^{-1}) $.

Now suppose the result is true for some $i \geq 1$, we shall prove it also holds for 
$i+1$.  We first let 
$$
u = (Q^{i})^{-1} \times \ldots \times  (Q^{1})^{-1} x.
$$
Since all the $Q^i$ are automorphisms of $\stdCone$, we have $u \in \interior \stdCone$.
Now, we argue that $\norm{u}_{1,\infty} \leq 1$. Since every $Q^i$ is constructed in a block-wise manner, we have 
$$
u_k = (Q^{i}_k)^{-1} \times \ldots \times  (Q^{1}_k)^{-1} x_k,
$$
which, by the induction hypothesis, implies that $u_k \in H(\stdInt_k,\stdInt_k/r_k)$ for every $k$. Therefore, $\inProd{u_k}{\stdInt _k} \leq 1$ for every $k$, which implies $\norm{u}_{1,\infty} \leq 1$.
Then, following Line \ref{alg:main:update_a}, $u$ is a feasible solution to the problem of finding $\hat x$ such 
that $\stdMap ^{i+1}\hat x = 0$, $\norm{\hat x}_{1,\infty} \leq 1$ and $\hat x \in \interior \stdCone$, where
$$
\stdMap ^{i+1} = \stdMap Q^{1} \times \cdots \times Q^i .
$$
As before, $Q^{i+1}_k$ is either the identity map or $Q_{w_k^{-1/2}\sqrt{r_k}}$. If it is the 
former, we are done. If it is the latter, it is because $w_k$ was constructed 
by applying Theorem \ref{theo:vol2} to $\stdMap ^{i+1},\stdCone$ and $y$, where $y$ is 
obtained at $(i+1)$-th iteration. We conclude  that 
$u _k \in H(w_k,w_k^{-1}) =  Q^{i+1}_kH(\stdInt_k, \stdInt_k/r_k)$. It follows that 
$(Q^{i+1}_k)^{-1}   u_k  \in H(\stdInt_k, \stdInt_k/r_k)$. This is equivalent to 
$x_k \in Q^{1}_k\times \cdots \times Q^{i+1}_k H(\stdInt_k,\stdInt_k/r_k).$
\end{proof}

The next lemma justifies the stopping criteria in Line \ref{alg:main:stop}.
\begin{lemma}\label{lemma:stop}
For all $x \in \feasPs$, for every iteration and for all $k \in \{1,\ldots, \ell\}$
$$\epsilon _k \geq \log r_k + \log \lambda _{\min}(x_k)
$$
holds. In particular, if $\epsilon _k <  \log r_k + \log \epsilon$,
then $\feasPs$ does not have an $\epsilon$-feasible solution.

\end{lemma}
\begin{proof}
We first handle the trivial case.
At Line \ref{alg:main:first}, $\epsilon _k$ is set to zero. Note that if $x \in \feasPs$, 
then $\inProd{x_k}{\stdInt _k} \leq 1$, from which it follows that $\lambda _{\min}(x_k) \leq 1/r_k$, 
since $\inProd{x_k}{\stdInt _k}$ is the sum of eigenvalues of $x_k$.
Therefore, $\log r_k + \log\lambda _{\min}(x_k)  \leq 0$.	
	
Lemma \ref{lemma:feas1} tells us that at the $i$-th iteration we have 
$x_k \in  Q_k H(\stdInt_k,\stdInt_k/r_k)$ for all $x \in \feasPs$, where $Q_k =  Q^1_k\times \cdots \times Q^i_k$.
Note that  $Q_k H(\stdInt_k,\stdInt_k/r_k)$ is also a half-space. 
Denote the volume of $(Q_k (H(\stdInt_k,\stdInt_k/r_k))) \cap \stdCone _k$ by $V_k$. 
Then, from Proposition \ref{prop:min_vol}, we have 
$V_k \geq r_k^{d_k} (\sqrt[r_k]{\det x_k})^{d_k} \Vol(\stdInt_k,\stdInt_k/r_k)$, for all 
$x \in \feasPs$. As $\det x_k$ is the product of eigenvalues of $x_k$, we get 
\begin{equation} \label{eq:log_bound}
V_k \geq  r_k^{d_k} \lambda _{\min}(x_k)^{d_k} \Vol(\stdInt_k,\stdInt_k/r_k), 
\end{equation}
for all $x \in \feasPs$. Furthermore, we have $V_k = \Vol(\stdInt_k,\stdInt_k/r_k) \prod _{j=1}^i{\det Q^j_k}$ and $\det Q^j_k$ is either $1$ or
$$ \left(\frac{r_k}{{\sqrt[r_k]{\det w_k}}} \right)^{d_k}
$$ depending on whether at the $j$-th iteration we 
had $\rho _k \leq 1$ or $\rho _k > 1$, respectively. Taking the logarithm at both
sides of \eqref{eq:log_bound}, we conclude that 
$$
\epsilon _k \geq \log r_k + \log \lambda _{\min}(x_k).
$$
Therefore, if $\epsilon _k < \log r_k + \log \epsilon $, we 
get $\epsilon > \lambda _{\min}(x_k) \geq \lambda _{\min}(x)$, for all 
$x \in \feasPs$. This implies the absence of $\epsilon$-feasible solutions.
\end{proof}	

For what follows, we discard the (trivial) case where $\epsilon$ is large.  Note that if $\epsilon \geq \frac{1}{r_k}$ for some $k$, then for $x \in \feasPs$
we have $\lambda _{\min}(x) \leq \lambda _{\min}(x_k) \leq \epsilon$, since the sum of the eigenvalues of $x_k$ is 
less or equal than $1/r_k$. In this case, there is no work to be done as this shows that $\feasPs$ has no $\epsilon$-feasible solution. So in the next result, we suppose 
that $\epsilon < \frac{1}{r_k}$ holds for every $k$.
\begin{theorem}\label{theo:main}
Algorithm \ref{alg:main} stops  after no more than 
$$
 \frac{r}{\varphi(2)}\log \left(\frac{1}{\epsilon}\right) -  \sum _{k= 1} ^\ell\frac{r_k \log(r_k)}{\varphi(2)}
$$ iterations, where $\varphi(\rho)$ is as in Theorem \ref{theo:vol2}. In particular, $ \frac{r}{\varphi(2)}\log \left(\frac{1}{\epsilon}\right)$ iterations 
are enough.
\end{theorem}
\begin{proof}
At each iteration, after Algorithm \ref{alg:basic} is called we obtain some point $y$. If 
$y$ is neither feasible for \eqref{eq:feas_p} nor \eqref{eq:feas_d}, then it satisfies 
$\norm{P_\stdMap  y} \leq \frac{1}{2r_{\max}\sqrt{\ell}}\norm{y}_{{1,\infty}}$.

Due to item $(ii)$ of Lemma \ref{lemma:y_proj2}, there is at least one 
index $k$ for which $\rho _k \geq 2$. Then from Theorem \ref{theo:vol2}, we have that
$$
\log r_k -  \frac{1}{r_k}\log {\det w_k} \leq -\frac{\varphi(2)}{r_k} < 0.
$$
This means that $\epsilon _k$ decrease by at least $\frac{\varphi(2)}{r_k}$.
Let us say that ``an iteration is good for $\stdCone_k$'' if at that iteration we 
have $\rho_k \geq 2$. From Lemma 
\ref{lemma:stop}, it follows that we need no more than
$$
\frac{r_k}{\varphi(2)}\log \left(\frac{1}{\epsilon r_{k}}\right).
$$
``good iterations for $\stdCone_k$'' before the stopping criteria in Line \ref{alg:main:stop} is 
satisfied. In particular, 
after
$$
\sum _{k= 1} ^\ell \frac{r_k}{\varphi(2)}\log \left(\frac{1}{\epsilon}\right) - \frac{r_k \log(r_k)}{\varphi(2)} =  \frac{r}{\varphi(2)}\log \left(\frac{1}{\epsilon}\right) -  \sum _{k= 1} ^\ell\frac{r_k \log(r_k)}{\varphi(2)} .
$$
iterations, we will meet the minimum number of good iterations for at least one cone.
We can then discard the  negative terms and obtain the bound $\frac{r}{\varphi(2)}\log \left(\frac{1}{\epsilon}\right)$
%
\end{proof}

We now take a look at the cost per iteration of Algorithm \ref{alg:main}. The two most 
expensive operations are computing $P_{\stdMap}$, calling the basic procedure (Line \ref{alg:main:call_bp}) and computing 
the square root of the inverse of $w_k$ (Line \ref{alg:main:qwk}). 
As discussed in Section \ref{sec:bp}, the cost of computing the projection $P_{\stdMap}$ is no more than $\mathcal{O}(m^3) + \mathcal{O}(m^2d)$.

Furthermore, when 
$\stdCone _k$ is $\Re _+, \SOC{d_k}$ or $\PSDcone{r_k}$, we need respectively no 
more than $\mathcal{O}(1), \mathcal{O}(d_k), \mathcal{O}(r_k^3) = \mathcal{O}(d_k^{3/2})$ for approximating the inverse of the square root of $w_k$. 
So, it seems that in most cases of interest it will cost no more than $\mathcal{O}(d_k^{3/2})$ to 
perform Line \ref{alg:main:qwk}. As we might need to compute $w_k^{-1/2}$ for all blocks, this 
will cost, in total, no more than $\mathcal{O}(d^{3/2})$.

Therefore, the cost of the basic procedure dominates the cost of 
computing $w_k^{-1/2}$, see \eqref{eq:bp:cost}. In total, we 
get 
$$
\mathcal{O}\left(\frac{r}{\varphi(2)}\log \left(\frac{1}{\epsilon}\right) ( m^3+m^2d  +\ell^3 r_{\max}^2(\max(md,c_{\text{min}})))\right),
$$
where $c_{\text{min}}$ is the cost of computing the minimum eigenvalue of an 
element $x \in \stdCone$. 

\section{Final remarks}\label{sec:conc}
In this work, we presented a generalization of Chubanov's algorithm to feasibility problems 
over symmetric cones.
A next interesting step would be to try to extend the algorithm to other families of 
cones. A key aspect of 
our algorithm is being able to shift the hyperplane $H(w,v)\cap \stdCone$ back to 
$H(\stdInt,\stdInt/r)$, with the aid of some appropriately constructed $Q_x$. The reason 
why this can always be done is because $\stdCone$ is a \emph{homogeneous cone}, which means that for every two points $x,y \in \interior \stdCone$ there is an automorphism $Q$ of 
$\stdCone$ such that $Q(x) = y$. It seems likely, although nontrivial, that Chubanov's algorithm can be extended to general 
homogeneous cone in some form or another. However, it less clear whether such an extension is 
possible for non-homogeneous cones. 

From a more practical side, there is a still a lack of computational results for semidefinite 
programming and second order cone programming. It will be interesting to take a look at whether 
Chubanov's algorithm could be competitive with IPMs codes when it comes to feasibility problems.

A significant bottleneck in the algorithm presented here is the computation 
of the projection matrix. Here, we are considering a simple approach where 
the  $P _{\stdMap ^i}$ is computed from scratch every time the 
basic procedure is called. One possible way to address this is using an 
incremental strategy analogous to what Pe\~na and Soheili suggested in 
Section 6 of \cite{PS16}. We plan to address this topic in detail in  future works.

There are also a number of intriguing aspects yet to be elucidated.  First of all, in our analysis, the self-concordance of the function $-\log \det$ plays a 
 very important role in guaranteeing that the reductions in volume are bounded by a constant. 
 This usage indicates that there could be a deeper connection to interior-point 
 methods (IPMs) then what is suggested by previous discussions.

Related to that, one of the referee's suggested that Chubanov's algorithm and our extension could be related to analytic center cutting plane methods, in particular the algorithms described by Sun, Toh and Zhao \cite{STZ02} and by Toh, Zhao and Sun \cite{STZ02,TZS02}, which solve feasibility problems over the positive semidefinite matrices. These methods were partly inspired by earlier work on linear feasibility problems by Luo and Sun \cite{LS00} and Goffin and Vial \cite{GV02}. They were later adapted to 
second order cone programming by Oskoorouchi and Goffin \cite{OG05} and  extended to symmetric cone programming by Basescu and Mitchell \cite{BM08}. 
Here we briefly explain the basic setting and try to explain the similarities and differences between our approach and analytic center methods.

We consider a convex body $\Gamma$  contained in $\stdCone$ that is only accessible by querying some oracle. Given some $x \in \jAlg$ the oracle either tell us that 
$x \in \Gamma$ or provide a cut $y$ such that $\Gamma$ is guaranteed to be contained in 
the half-space $H_1 = \{z \in \jAlg \mid \inProd{z}{y} \leq \inProd{x}{y} \} $.
We also assume that $\Gamma$ is contained in the ``spectral interval'' $\Omega _0 = \{z \in \jAlg \mid z \in \stdCone, \stdInt - z \in \stdCone \}$ and that it contains some 
$\epsilon$-ball. Note that these assumptions imply $\Gamma \cap \interior \stdCone \neq\emptyset$ and that $\Gamma$ is full-dimensional. 
A basic analytic center cutting plane algorithm would proceed by first guessing some initial point $x_0$. Then, the oracle is queried and either $x_0$ is feasible or a cut is returned. If a cut is returned, then we let $x_1$ be the analytic center of $\Omega_1 \coloneqq \Omega _0 \cap H_0$ or some appropriate approximation. Then, we test  $x_1$ for feasibility and if we receive a cut $y_1$ from the oracle, we let $x_2$ be the analytic center of $\Omega _2 \coloneqq \Omega _0 \cap H_0\cap H_1$ or some approximation. We then repeat until a feasible solution is found or some other stopping criterion is met. 
Here, we recall that the analytic center is the minimizer of a function that is constructed using the barrier $-\log \det(\cdot)$, see Section 2 in \cite{STZ02}.

As for the similarities and differences, first of all, it seems that the setting is slightly different. For example, neither the feasible region of \eqref{eq:feas_p} nor \eqref{eq:feas_p_s} can be expected to contain some $\epsilon$-ball. 
Nevertheless, the basic procedure could be seen as an oracle 
for the feasible region of  \eqref{eq:feas_p_s}, in the sense that it returns a cut $y$ when the procedure is not able to prove or disprove the feasibility of \eqref{eq:feas_p_s}. This cut can be used to define half-spaces that contain some of the blocks of the feasible region of \eqref{eq:feas_p_s}, see item $(ii)$ of Lemma \ref{lemma:y_proj2}.
However, the basic procedure is more than a simple oracle and proactively tries to improve the point received as input.
Then, after $y$ is found we try to improve it by decreasing a function constructed around $-\log \det(\cdot)$, as in Theorem \ref{theo:vol2}.
It might be possible that the update in Theorem \ref{theo:vol2} could be seen as trying to approximate the analytic center of some set that contains the feasible region of 
\eqref{eq:feas_p_s}, but, at this moment, it is not clear to us the appropriate way of establishing such a connection.
Furthermore, the step where $\stdMap^i$ is updated to $\stdMap^iQ^i$ (Line \ref{alg:main:update_a} in Algorithm \ref{alg:main}) does not seem to have a clear counterpart in analytic center methods.
Nevertheless, we think it could be an interesting topic of future research to take a closer look at the relationship between the methods.

\small{
\section*{Acknowledgements}
We thank the referees for their helpful and insightful comments, which helped to improve the paper. 	
This article benefited from an e-mail discussion with Prof. Javier Pe\~{n}a, which helped clarify some points regarding \cite{PS16}. 
T.\@ Kitahara is supported by Grant-in-Aid for Young Scientists (B) 15K15941.
M.\@ Muramatsu and T.\@ Tsuchiya are supported in part with Grant-in-Aid for Scientific Research (B)24310112 and 
(C) 26330025. M.\@ Muramatsu is  also partially supported by the
Grant-in-Aid for Scientific Research (B)26280005. T.\@ Tsuchiya is also partially supported by the Grant-in-Aid for Scientific Research (B)15H02968.
}
\bibliographystyle{abbrvurl}
\bibliography{bib_plain}

\appendix
\section{The case of semidefinite programming}\label{app:sdp}
For ease of reference, we ``translate'' here  Algorithms  
\ref{alg:basic} and \ref{alg:main} for the case of semidefinite 
programming, see Algorithms \ref{alg:basic:sdp} and \ref{alg:main:sdp}. In this case, we have $\ell = 1, r = n$, $\stdCone = \PSDcone{n}$ and 
$\jAlg$ is the space of $n\times n$ symmetric matrices $\S^n$.
We recall that if $y \in \PSDcone{n}$, then $\norm{y}_1 = \inProd{\stdInt}{y} = \tr y$, so 
the stopping criteria in Algorithm \ref{alg:basic} becomes 
$\norm{z} \leq \frac{1}{2n}\tr y$. We will use $I_n$ to denote the $n\times n$ identity matrix. We will write $y \succeq 0$, if $y$ is positive semidefinite and $y \succ 0$, if $y$ is positive definite.

Furthermore, for $w \in \PSDcone{n}$, the quadratic map $Q_{w}$ is 
the function that takes $x \in \S^n$ to $wxw$. We do not need, in fact, to explicitly 
construct the operator $Q_{w}$. In particular, if $\stdMap x = 0$ is represented as the set of solutions of $m$ linear equalities $\tr(a_1x) = 0, \ldots, \tr(a_mx) = 0$, the first assignment in Line \ref{alg:main:sdp:last} of Algorithm \ref{alg:main:sdp} is the same 
as substituting every $a_i$ by $n(w^{-1/2}a_iw^{-1/2})$.
Finally, as remarked at the end of Section \ref{sec:bp}, the condition 
$z \neq 0$ in Algorithm \ref{alg:basic:sdp} can be substituted by 
$y^i - z \not \succeq 0$.

{
\begin{algorithm}[H]
\small
	\caption{Basic Procedure - SDP }\label{alg:basic:sdp}
	\KwIn{$\stdMap$, $P_{\stdMap}$, $y^1 \succ 0$ such that $\tr y^1= 1$}
	\KwOut{
		\begin{enumerate}
			\item $y \succeq 0$ such that $\norm{P_\stdMap  y} \leq \frac{1}{2n}\norm{y}_{{1}}$, or
			\item a solution to \eqref{eq:feas_p}, or
			\item a solution to \eqref{eq:feas_d}.
		\end{enumerate}}
		
		$ i \leftarrow 1$, $z \leftarrow P_\stdMap  y^1$. \label{alg:basic:sdp:proj}\\
		\While{$z \neq 0$ and $z \not \succ 0 $ and $\norm{z} > \frac{1}{2n}\tr y^i$}{
			Let $v$ be an eigenvector associated to the minimum eigenvalue of $z$ and such that $\norm{v} = 1$. Let $c \leftarrow vv^\T$. \label{alg:basic:sdp:ide}\\
			$p \leftarrow P_\stdMap c$.\\
			\eIf{$p = 0$ or $p \succ 0$ or $\norm{p} \leq \frac{1}{2n} $}
			{\textbf{stop} and \textbf{return} $c$.}
			{$y^{i+1} \leftarrow \alpha y^i + (1-\alpha) c$, where $\alpha=\frac{\inProd{p}{p - z}}{\norm{z-p}^2}$. \\}
			Set $i \leftarrow i+1$ and $z \leftarrow P_\stdMap  y^i$.
		}
		\textbf{return} $y^i$.

	\end{algorithm}

\begin{algorithm}
\small
	\caption{Main Algorithm - SDP}\label{alg:main:sdp}
	\KwIn{$\stdMap$, $\PSDcone{n}$, $\epsilon$}
	\KwOut{\begin{enumerate}
			\item a solution to \eqref{eq:feas_p}, or
			\item a solution to \eqref{eq:feas_d}, or
			\item ``there is no $\epsilon$-feasible solution''.
		\end{enumerate}}
	$\stdMap ^1 \leftarrow \stdMap$, $i \leftarrow 1$ and 
	$\tilde \epsilon \leftarrow 0$.\\
	Compute $P_{\stdMap}$ and call the basic procedure (Algorithm \ref{alg:basic:sdp}) with $\stdMap^i, P_{\stdMap}, \frac{I_n}{n}$ and 
	denote its output by $y$. \label{alg:main:sdp:call_bp}\\
	\If{$y$ is feasible for \eqref{eq:feas_d} or \eqref{eq:feas_p}}{\textbf{stop} and \textbf{return} $y$.\label{alg:main:sdp:stop1}}
	$\rho  \leftarrow \frac{\tr y}{n\norm{P_\stdMap  y}}$. \label{alg:main:sdp:rho}\\
	$\beta \leftarrow n- \left(\frac{1}{\rho} - \frac{1}{\sqrt{n(3n - 2)}}\right)$.\\
	$w \leftarrow  \left(\frac{n-\beta}{\tr y} \right) \rho n y + \beta I_n$.\label{alg:main:sdp:wk}\\
	$\tilde \epsilon  \leftarrow \tilde \epsilon + \log n -  \frac{1}{n}\log {\det w} $.\\
	\If{$\tilde \epsilon  < \log n + \log \epsilon$ \label{alg:main:sdp:stop}}{\textbf{stop}, there is no $\epsilon$-feasible solution. }
	$\stdMap^{i+1} \leftarrow \stdMap^i Q_{w^{-1/2}\sqrt{n}},$ $i\leftarrow i+1$. Go to Line \ref{alg:main:sdp:call_bp}. \label{alg:main:sdp:last}
\end{algorithm}
}

\end{document}